\documentclass[reqno,12pt]{amsart}

\usepackage{color} 
\usepackage{ifpdf}
\ifpdf
    \usepackage[pdftex]{graphicx}
    \usepackage[pdftex]{hyperref}
    \hypersetup{
        unicode=false,          
        pdftoolbar=true,        
        pdfmenubar=true,        
        pdffitwindow=false,     
        pdfstartview={FitH},    
        pdftitle={MCP Article},      
        pdfauthor={Sara Pollock},   
        pdfsubject={Mathematics},    
        pdfcreator={Sara Pollock},  
        pdfproducer={Sara Pollock}, 
        pdfkeywords={PDE, analysis, numerical analysis}, 
        pdfnewwindow=true,      
        colorlinks=true,        
        linkcolor=red,          
        citecolor=blue,         
        filecolor=magenta,      
        urlcolor=cyan           
    }

\else
    \usepackage{graphicx}
    \usepackage{pstricks}
    
    \newcommand{\href}[2]{#2}
\fi

\usepackage{times}
\usepackage{amsfonts}
\usepackage[font={small}]{caption}

\usepackage{amsmath}
\usepackage{amsthm}
\usepackage{amssymb}
\usepackage{amsbsy}
\usepackage{amscd}
\usepackage{extarrows}

\usepackage{enumerate}
\usepackage{verbatim}
\usepackage{subfigure}

\newtheorem{theorem}{Theorem}[section]
\newtheorem{corollary}[theorem]{Corollary}
\newtheorem{lemma}[theorem]{Lemma}

\newtheorem{assumption}[theorem]{Assumption}
\newtheorem{proposition}[theorem]{Proposition}

\newtheorem{definition}[theorem]{Definition}
\newtheorem{example}[theorem]{Example}
\newtheorem{remark}[theorem]{Remark}

\newtheorem{algorithm}[theorem]{Algorithm}
\newtheorem{criteria}[theorem]{Criteria}

\numberwithin{equation}{section}  

  \newcounter{mnote}
  \let\oldmarginpar\marginpar
    \renewcommand\marginpar[1]{\-\oldmarginpar[\raggedleft\footnotesize #1]%
    {\raggedright\footnotesize #1}}

\newenvironment{enumerateY}
{\begin{list}{{\it(\roman{enumii})}}{
\usecounter{enumii}
\leftmargin 2.5em\topsep 0.5em\itemsep -0.0em\labelwidth 50.0em}}
{\end{list}}

\newenvironment{enumerateX}
{\begin{list}{\arabic{enumi})}{
\usecounter{enumi}
\leftmargin 2.5em\topsep 0.5em\itemsep -0.0em\labelwidth 50.0em}}
{\end{list}}

\newenvironment{itemizeX}
{\begin{list}{\labelitemi}{
\leftmargin 2.5em\topsep 0.5em\itemsep -0.0em\labelwidth 50.0em}}
{\end{list}}

\definecolor{myblue}{rgb}{0.2,0.2,0.7}
\definecolor{mygreen}{rgb}{0,0.6,0}
\definecolor{mycyan}{rgb}{0,0.6,0.6}
\definecolor{myred}{rgb}{0.9,0.2,0.2}
\definecolor{mymagenta}{rgb}{0.9,0.2,0.9}
\definecolor{mywhite}{rgb}{1.0,1.0,1.0}
\definecolor{myblack}{rgb}{0.0,0.0,0.0}

\renewcommand{\div}{{\operatorname{div}}}
\newcommand{\eps}{\varepsilon}

\newcommand{\GM}{\gamma_{\text{MAX}}}
\newcommand{\gmo}{\gamma_{\text{MONO}}}
\newcommand{\dmin}{\delta_{\text{MIN}}}
\newcommand{\econ}{\eps_{\text{CON}}}
\newcommand{\imax}{I_{\text{MAX}}}
\newcommand{\ibase}{I_{\text{BASE}}}
\newcommand{\iacc}{I_{\text{ACC}}}
\newcommand{\qacc}{q_{\text{ACC}}}

\newcommand{\PP}{{\mathbb P}}       
\newcommand{\R}{{\mathbb R}}       

\newcommand{\cL}{{\mathcal L}}
\newcommand{\cM}{{\mathcal M}}

\newcommand{\cS}{{\mathcal S}}
\newcommand{\cT}{{\mathcal T}}

\DeclareMathAlphabet{\mathpzc}{OT1}{pzc}{m}{it}

\newcommand{\f}{\frac}

\newcommand{\forr}{\text{ for }}

\newcommand{\on}{\text{ on }}

\newcommand{\an}{\text{ and }}

\newcommand{\with}{\text{\ with }}

\newcommand{\inn}{\text{ in }}

\newcommand{\tforall}{\text{ for all }}

\newcommand{\rest}{\big|}
\newcommand{\Rest}{\Big|}

\newcommand{\grad}{\nabla} 

\newcommand{\goto}{\rightarrow}

\newcommand{\norm}[1]{\ensuremath{\lVert{#1} \rVert}}
\newcommand{\nr}[1]{\norm{#1}} 

\newcommand{\anorm}{\norm{\ \cdot \ }}

\newcommand{\pa}{\partial}

\definecolor{blue}{rgb}{0.2,0.2,0.7}
\definecolor{red}{rgb}{0.7,0.3,0.1}
\definecolor{cyan}{rgb}{0.2,0.5,0.6}
\usepackage{mathtools}
\usepackage{stmaryrd}

\begin{document}

\title[Layer capturing adaptive methods]
      {Stabilized and inexact adaptive methods for capturing internal layers
in quasilinear PDE}

\author[S. Pollock]{Sara Pollock}
\email{snpolloc@math.tamu.edu}

\address{Department of Mathematics\\
         Texas A\&M University\\ 
         College Station, TX 77843}


\date{\today}



\begin{abstract}
A method is developed within an adaptive framework to solve quasilinear
diffusion problems with internal and possibly boundary layers starting from
a coarse mesh.
The solution process is assumed to start on a mesh where the problem
is badly resolved, and approximation properties of the exact problem and its
corresponding finite element solution do not hold. A
sequence of stabilized and inexact  partial solves allow the mesh to be 
refined to capture internal layers while an approximate solution is 
built eventually leading to an accurate approximation of both the problem 
and its solution.  The innovations in the current work include a closed form
definition for the numerical dissipation and inexact scaling parameters on 
each mesh refinement, as well as a convergence result for the residual of 
the discrete problem.  Numerical experiments demonstrate the method on a range
of problems featuring steep internal layers and high solution dependent 
frequencies of the diffusion coefficients.
\end{abstract}

\maketitle




\section{Introduction}
This investigation continues the work in~\cite{Pollock14a,Pollock15a}
developing adaptive numerical methods for quasilinear partial differential
equations featuring steep
internal layers, in which the solution process
starts on a coarse mesh where the problem is not yet resolved.
Throughout the coarse mesh and preasymptotic regimes, standard methods such 
as Newton iterations are known to fail due to both the ill-conditioned and 
possibly indefinite Jacobians which are characteristic of the approximate
discrete problems, and the partial resolution of the problem data.
This paper
specifies an appropriate set of parameters that may be used in the
stabilized 
$\sigma$-Newmark strategy of~\cite{Pollock15a} 
applied to quasilinear diffusion problems, where the layers develop 
from both the 
solution dependent coefficients and a variable dependent source. 
For the class of
problems studied here 
\begin{align}\label{eqn:introclass1}
-\div(\kappa(u)\grad u) - f(x) &= 0, ~u = 0 \on \pa \Omega, \quad \an
\\ \label{eqn:introclass2} 
-\div(\kappa(|\grad u|^2) \grad u) - f(x) &= 0, ~u = 0 \on \pa \Omega,
\end{align}
in which the diffusion is bounded away from zero,
local uniqueness of the solution is known, as well as approximation properties
for the finite element solution using linear elements~\cite{Xu96}, assuming
the mesh is sufficiently fine.  
For operators containing steep solution-dependent layers in the coefficients, the
approximation properties of the discrete solution are useful only if the 
solution to the discrete nonlinear problem can be attained, and the current 
method attains such a solution by means of a sequence of approximate problems
with inexact source functions that limit to the discrete problem.   

The methodology is to first discretize 
~\eqref{eqn:introclass1} and~\eqref{eqn:introclass2} on each mesh refinement
then linearize the resulting discrete problem, and to partially 
solve the sequence of resulting inaccurate and ill-conditioned
coarse mesh problems by stabilized Newton-like iterations
while adaptively refining the mesh leading to
an accurate and efficient solve of an accurate discretization of the problem.
In previous
work the focus was on the stabilization of the Jacobian by a combination of 
regularization and added numerical dissipation.  Currently, a formula for 
the numerical dissipation parameter is presented, along with a new inexact method
designed for problems where the variable-dependent source dominates the
residual of the Newton-like iterations.

Iterative rescaling techniques in the solution of nonlinear problems are not 
uncommon, see for instance the scaling iterative algorithm (SIA) of 
~\cite{ChZhNi00} for the solution of semilinear elliptic problems, in which the
solution $u$ is rescaled at each iteration. 
The rescaling of the Monge-Amp\`ere equation in ~\cite{Awanou14a} to establish a 
fixed-point argument and recover a numerical solution without having to
assume the solution is small enough motivated the current approach.
Here, the inexact method 
rescales the variable dependent source until the solution iterates
attain sufficient stability to solve for the given data.   

Recent approaches such as~\cite{AlErVo11,ErnVor13} for
monotone quasilinear problems use 
inexact linear and nonlinear solves to avoid over-solving for the residual when
the Galerkin or discretization error are the dominant sources of error.
It is assumed it their analysis that the discrete problem on each 
refinement is well posed.  
In the problems studied here, the coarse mesh problem may not be well posed, 
and may be a  
sufficiently bad approximation of the exact problem that estimates of the 
different error sources are not necessarily well determined or useful.  
Instead, the iterations are ended when they stabilize to the predicted linear 
convergence rate, which is a function of the numerical dissipation parameter.  
Combined with the criteria that the residual from the linear solves on each mesh 
refinement must show sufficient decrease with respect to the residual on the 
previous refinement, the sequence of stabilized and inexact problems recovers
the unscaled discrete problem on a mesh where it is better represented, and with
an initial guess for the Newton-like iterations that yields the 
discrete problem solvable. 
This method predicts the stability of the solve and allows the sequence 
of coarse mesh
problems to be solved approximately through the preasymptotic regime leading 
to an efficient solve in the asymptotic regime.
  
The remainder of the paper is organized as follows.   
Section~\ref{sec:problem_class} states the target problem class and the
formulation of the discrete problems. 
Section~\ref{sec:stableadaptive} reviews the Jacobian stabilization techniques
developed by the author in previous work, and which are further developed here.
Section~\ref{sec:gamma_update} presents a formulation for the numerical
dissipation parameter $\gamma$ and characterizes its properties within the
adaptive framework; then Section~\ref{sec:delta_update} presents a formulation
for the inexact scaling parameter $\delta$ and characterizes its convergence to
unity within the adaptive method.  Section~\ref{sec:algorithm} summarizes the
results of the previous three sections into an adaptive algorithm and proves 
the convergence of the residual of the discrete problem.  Finally, 
Section~\ref{sec:numerics} demonstrates the method with a collection of
numerical experiments featuring different types of internal layers.

The following notation is used throughout the rest of the paper. 
In defining the weak and bilinear forms in the next section 
$(u(x) , v(x) ) = \int_\Omega u(x) v(x)\, dx$, and in later section 
the discrete inner product between vectors $u_k, v_k \in \R^n$ is  denoted
$\langle u_k, v_k \rangle$.
The norm $\anorm$ where not otherwise specified is the $L_2$ norm.
The $n$th iterate subordinate to the $k$th partition $\cT_k$ is denoted 
$u^n_k$, while $u^n$ is the $n$th iteration on a fixed partition and $u_k$ is 
the final iteration on the $k$th mesh, 
taken as the approximate solution on $\cT_k$.  
\section{Target problem class}\label{sec:problem_class}
The class of problems considered are quasilinear diffusion problems 
$F(u,x)= 0$, over polygonal domain $\Omega \subset \R^2$,
with $F: X \times \Omega 
\goto Y^\ast$ and $F'(u,x) \coloneqq F_u(u,x)  \in \cL(X, Y^\ast)$, 
where $F(u,x)$
is given by
\begin{align}\label{eqn:quasiclass1}
F(u,x)& \coloneqq -\div(\kappa(u) \grad u) - f(x) = 0, 
\inn \Omega \in \R^2, ~ u = 0 \on \pa \Omega, \text{ or } 
\\ \label{eqn:quasiclass2}
F(u,x) & \coloneqq-\div(\kappa(|\grad u|^2) \grad u) - f(x)=0, 
\inn \Omega \in \R^2,~ u = 0 \on \pa \Omega,
\end{align}
with $f(x) \in L_2(\Omega)\cap L_\infty(\Omega)$.
Multiplication against test functions $v \in Y$ and integration by parts
yields the weak form of each problem
\begin{align}\label{eqn:weakform1}
B(u,v) & = (\kappa(u) \grad u, \grad v) = (f,v), ~\tforall v \in Y, 
&\forr~\eqref{eqn:quasiclass1}, 
\\ \label{eqn:weakform2}
B(u,v) & = (\kappa(|\grad u|^2)\grad u, \grad v) = (f,v) ~\tforall v \in Y,
&\forr~\eqref{eqn:quasiclass2}.
\end{align}
The linearized form induced by $F'(u,x)\coloneqq F_u(u,x)$, is determined by taking the 
Gateaux derivative in 
direction $w \in X$ by $B'(u; w, v) =  \lim_{t \goto 0}\, d/dt\, ( B(u + tw, v))$ 
yielding 
\begin{align}\label{linearform1}
B'(u;w,v) & = (\kappa(u) \grad w, \grad v) + (\kappa'(u) w \grad u, \grad v), 
&\forr~\eqref{eqn:quasiclass1}, 
\\ \label{eqn:linearform2}
B'(u;w,v) & = (\kappa(|\grad u|^2)\grad w, \grad v) + 
           (\kappa'(|\grad u|^2)(2 \grad u\cdot \grad w) \grad u, \grad v),
&\forr~\eqref{eqn:quasiclass2}.
\end{align}
Both types of problems fit into the context of~\cite{Xu96} with the 
assumption that there is a solution $u \in H_0^1(\Omega) \cap W_{2+\eps}^2(\Omega)$
and
$F_u(u,x): H_0^1(\Omega) \goto H^{-1}(\Omega)$ is an isomorphism, 
in which case the solution $u$ is an isolated solution, and approximation
properties for the linear Lagrange finite element solution can be shown to hold, 
assuming the meshsize is fine enough.

It is further shown in~\cite{CaRa94} 
for problems of the form~\eqref{eqn:quasiclass1}, assuming 
$\kappa(s) \ge k > 0$ for all $s \in \R$ and $\kappa(s), \kappa'(s),
\kappa''(s)$ bounded, 
then there exists a unique solution 
$u \in W^{1,p}(\Omega)$, with $2 < p < \infty$.  
Moreover, by the analysis of~\cite{HTZ09a}, an adaptive method using 
linear Lagrange finite elements can be shown to converge, again assuming
the meshsize is sufficiently small.

The second class of equations~\eqref{eqn:quasiclass2} is considered in
for instance~\cite{GaMoZu11,BDK12,AlErVo11} and the reference therein, 
and uniqueness of solutions as well as convergence of
adaptive methods have been established supposing a 
Lipchitz condition and a monotonicity
condition such as $B(v,v-w) - B(w,v-w) \ge 0$ or the strong monotonicity 
condition $B(v,v-w) - B(w,v-w) \ge c \nr{v-w}^2_Y$ as in~\cite{GaMoZu11},
and assuming the approximate discrete problems are well posed.

\subsection{Discrete formulation}\label{subsec:discreteform}
The discrete problem is, find $u_h \in X_h$ such that
$B(u_h,v)=0 \tforall v \in Y_h$ where $X_h \subset X$ and $Y_h \subset Y$
are discrete spaces of the same finite dimension.
Here, $X_h$ and $Y_h$ are finite element spaces with respect to triangulation 
$\cT_h$, 
where the family of triangulations $\{\cT_h\}_{0 < h < 1}$ is regular and 
quasi-uniform  in the sense of~\cite{Ciarlet}. 
For the rest of the discussion it is assumed the trial and test spaces
are the same: $X = Y$ and $X_h = Y_h$. 
On each refinement $k$, the test/trial space $X_{h_k}$  
is taken as the linear finite element space $V_k$ consisting of 
Lagrange finite elements  $\PP_1$ over partition $\cT_k$ that satisfy 
the homogeneous Dirichlet boundary conditions.

The resulting discrete problem on refinement $k$ 
is of the form $G_k(u,x)= 0$ for $u \in V_k$ and $x \in \Omega \subset \R^2$.
The decomposition $G_k(u,x) = -g(u) + f_k(x)$ is used throughout the remainder 
of the discussion.   
On a given discretization that is not sufficiently fine, $f_k(x)$ may be a 
bad representation of the given source $f(x)$. On such a coarse mesh, solving  
the discretized equation $g(u_k) = f_k(x)$ may produce a solution 
$u_k$ which is a highly inaccurate solution to ~\eqref{eqn:weakform1}, or
respectively~\eqref{eqn:weakform2}, so seeking an exact solution 
to the approximate problem is
not a good use of computational effort.   

Ideally, the adaptive method starts on
a coarse mesh where the problem is small, if badly resolved. It produces
a sequence of approximations that are good enough to refine the mesh to resolve
the data and solution dependent coefficients, and upon sufficient resolution
and stability,
provide a good approximation for the solution to
the  discrete problem.

The features of the current method as introduced in~\cite{Pollock14a, Pollock15a}
include an adaptive approach to regularization used to stabilize the Jacobian in 
the coarse mesh and preasymptotic regimes, also conceived of as 
adaptively regulated implicit pseudo-time stepping.  

\section{Stabilized adaptive method}\label{sec:stableadaptive} 
The stabilized iterations are designed
to fit into the framework of a standard adaptive method, which uses error 
indicators to mark and refine the mesh.
The theoretical results of this paper are independent of the specifics of
the adaptive method, and a different set of error indicators could be used.  
The method here is demonstrated using standard residual based estimators,
as in ~\cite{FiVe03,Stevenson07} for mesh refinement,
and similarly for the adaptive regularization strategy.  
Both the interior and the jump part of the residual indicator
are used to refine the mesh, and a similar jump indicator based only on the 
solution is used to determine  
the regularization.
The indicators are defined with respect to the consistent 
problem with unscaled data. 
\subsection{Adaptive mesh refinement}\label{subsec:adaptiverefine}
The local \emph{a posteriori} residual-based 
indicator, and corresponding jump-based indicator for element 
$T \in \cT_k$ with $h_T$ the element diameter are given by
\begin{align}
\label{eqn:jump_indicators}
\zeta_T^2(v) &= \zeta_{\cT_k}^2(v,T) \coloneqq h_T \nr{ J_T(v)  }_{L_2(\pa T)}^2,
\\ \label{eqn:indicators}
\eta_T^2(v) &= \eta_{\cT_k}^2(v,T) \coloneqq h_T^2 \nr{F(v,x)}_{L_2(T)}^2 
+  \zeta_T^2(v),
\end{align}
$J_T(v) \coloneqq \llbracket [\kappa(v) \grad v  \cdot n \rrbracket_{\pa T}$,
for problem~\eqref{eqn:quasiclass1} and
$J_T(v) \coloneqq \llbracket [\kappa(|\grad v|^2) \grad v  \cdot n \rrbracket_{\pa T}$, for problem~\eqref{eqn:quasiclass2}, 
with jump
$\llbracket \phi \rrbracket_{\pa T} 
\coloneqq {\lim_{t \goto 0} \phi(x + t n) - \phi(x - tn)}$,
 where  $n$ is the appropriate outward normal defined on $\pa T$.  
The error estimator on partition $\cT_k$ is given by the $l_2$ sum of 
indicators $\eta_{\cT_k}^2 = \sum_{T \in \cT_k}\eta_{T}^2$, and similarly 
for $\zeta_{\cT_k}$.

In these results, the D\"orfler marking strategy is used.  Given a 
parameter $\theta \in (0,1)$ a set of least, or nearly least cardinality
is chosen for the marked set $\cM$ so the sum of their indicators
is greater than the given fraction of the sum of all the indicators. 
\begin{align}\label{eqn:dorflermark}
\sum_{T \in \cM}\eta_T^2 \ge \theta \sum_{T \in \cT_k} \eta_T^2.
\end{align}
\begin{remark}\label{remark:coarsemarking}In contrast to an 
earlier version of the method for convection diffusion problems 
in~\cite{Pollock15a}, no additional coarse mesh marking is performed.
Scaling down the variable dependent source on the initial solves
allows the method to traverse
the preasymptotic and coarse mesh regimes while minimizing refinement
that negatively impacts the efficiency of the method in the
asymptotic regime. 
\end{remark}

\subsection{Adaptive regularization}\label{subsec:adaptivereg}
The motivation from multiple perspectives
behind the adaptive regularization technique from 
is described in~\cite{Pollock14a}, and an 
updated strategy is suggested in~\cite{Pollock15a}, which is summarized below.
The coarse mesh problems may have indefinite, and often have highly 
ill-conditioned Jacobians, and the Laplacian-based regularization 
helps make the Newton-like
iterations solvable.  The adaptive element is introduced to yield faster 
convergence as the degrees of freedom for which the approximate solution 
is smooth experience little benefit from the added
stabilization as compared to the degrees of freedom (dofs) 
in regions where the jump in the 
normal derivative between neighboring elements is large. 
On each mesh partition $\cT_k$ the penalty matrix $R = R_{k}$ is a 
localized version of the Laplacian stiffness matrix denoted $R^{global}$.
First define the jump indicator used to select elements for regularization,
{\em c.f.,} \eqref{eqn:jump_indicators}.
\begin{equation}\label{eqn:regjumpindic}
\xi_T^2(v) = h_T \nr{\llbracket \grad v \cdot n\rrbracket_{\pa T}}_{L_2(\pa T)}^2.
\end{equation}
Using the jump indicators $\{\xi_T\}_{T \in \cT}$, select degrees of freedom
for regularization.
\begin{definition}\label{def:Rk}  
Define $R^{global}$ as the Laplacian stiffness matrix on refinement $k$.
Let $D^{loc}$ a diagonal matrix of zeros and ones, $v_j$ a vertex  subordinate 
to partition $\cT_k$, and $\xi_T = \xi_T(u^0_k)$. Then set
\[
\tilde \psi_k = \sqrt{\text{median}_{T \in \cT_k}(\xi_T^2)}, ~\an
\psi_k = \left\{ \begin{array}{ll}
{\tilde \psi_k}^{1/2} , & \text{ if } {\tilde \psi_k} > 1, \\
 \tilde \psi, & \text{ otherwise},
\end{array}
\right.
\]
and
\begin{align}\label{eqn:adaptivereg}
D_{jj }^{loc} = \left\{ \begin{array}{ll}
 0, & \text{ if } \xi_T \le \psi_k \text{ for each element that contains } v_j
\text{ as a vertex}, \\
1, & \text{ otherwise}.
\end{array} \right.
\end{align}
Set $R_k = D^{loc}R^{global}D^{loc}$.
\end{definition}
The regularization matrix $R_k$ defined by Definition~\ref{def:Rk} is 
positive semidefinite and targeted to smooth regions of high curvature.

\subsection{Exact and inexact $\sigma$-Newmark iterations.}\label{sec:iterations}
The $\sigma$-Newark iteration of~\cite{Pollock15a} generalizes the 
pseudo-transient continuation method of~\cite{KeKe98,CoKeKe02,FoKe05} and 
similarly the ``s" method of ~\cite{BaRo80a} to an adaptive framework, combining
stabilization by pseudo-time with the Tikhonov-type regularization as used in ill-posed problems~\cite{Engl1996}.

The Jacobian is stabilized by seeking a steady-state solution of
the given nonlinear problem $G_k(u,x) = 0$ by solving 
$\pa (Ru)/ \pa t + G_k(u,x) = 0$, and discretizing the time derivative 
$\dot u \coloneqq \pa u / \pa t$
by a Newmark update~\cite{Newmark59} given by the time discretization
$u^{n+1} = u^n + \Delta t^n \{\gamma \dot u^{n+1}  + (1 - \gamma) \dot u^n\}$, or
solving for $\dot u^{n+1}$
\[
\dot u^{n+1} = \f{u^{n+1}-u^n}{\gamma \Delta t^n} 
- \f{(1 - \gamma)}{\gamma} \dot u^n.
\]
In the current notation the regularization parameter $\alpha^n = 1/\Delta t^n$.
As discussed in the literature on the finite element analysis of structures,
for instance~\cite{HiHu78a,ChHu93} and the references therein, the update 
yields second-order accuracy in time for $\gamma = 1/2$ and increased numerical
dissipation of the high frequency modes with $\gamma > 1/2$.  The current adaptive
method allows $\gamma> 1$, which can be seen to damp the influence of the
right-hand side residual.  In~\cite{Pollock15a} a secondary linearization is
performed to further stabilize the approximate Jacobian resulting in the iteration
\begin{align}\label{eqn:stabit}
\left(\alpha^n R + \gamma^n \{\sigma^n G'(u^n) + (1-\sigma^n) 
G'(\bar u) \} \right)w^n
= -G(u^n,x), \quad u^{n+1} = u^n + w^n.
\end{align}
In previous discussions the variable dependence of $G(u,x)$ was suppressed and
written $G(u)$.  In the current presentation the variable dependent load will
be stated explicitly by $G(u,x) = g(u) - f(x)$.  As the Jacobian $G_u(u,x)
= g'(u)$ is invariant to $f(x)$, the stabilization based on the Jacobian alone
does not address sharp spikes or steep gradient which may be present in $f(x)$, 
particularly in partially resolved coarse mesh problems.  The inexact iteration
is then introduced
\begin{align}\label{eqn:inexactit}
\left(\alpha^n R + \gamma^n \{\sigma^n g'(u^n) + (1-\sigma^n) 
g'(\bar u) \} \right)w^n
= -g(u^n) + \delta f(x), \quad u^{n+1} = u^n + w^n.
\end{align}
Iteration~\eqref{eqn:stabit} is defined by the following parameters on 
iteration $n$ of
refinement $k$. The limiting behaviors necessary to reduce the stabilized
iteration to the standard Newton iteration $g'(u) w^n = -g(u^n) + f(x)$
are noted. 
\begin{itemizeX}
\item Regularization parameter $\alpha = \alpha_k^n \ge 0, ~\alpha \goto 0$.
\item Semidefinite regularization term $R = R_k.$
\item Newmark parameter $\gamma = \gamma^n_k \ge 1, ~ \gamma^n_k \goto 1$.
\item Stabilization parameter $\sigma = \sigma^n_k \in [\sigma_0, 1], 
~ \sigma \goto 1$.
\end{itemizeX}
In addition, iteration~\eqref{eqn:inexactit} depends on the scaling parameter
\[
\delta = \delta_k \le 1, ~ \delta \goto 1.
\]
The focus of this paper is the description of an appropriate choice of parameters
$\gamma \an \delta$ to use in iterations~\eqref{eqn:stabit} and 
\eqref{eqn:inexactit}.  The regularization parameters $\alpha \an R$ are 
presented in~\cite{Pollock15a}, and recalled 
in this section.
The convergence behavior 
of iteration~\eqref{eqn:stabit} is analyzed in~\cite{Pollock15a}, and those 
results are recalled below.  
Definitions of $\gamma$ and $\delta$ are introduced in the following 
sections, as
well as a convergence result for the sequence of inexact iterations using
~\eqref{eqn:inexactit}.  

The following concepts are used throughout the discussion.
On a fixed refinement $k$, the subscript $k$ is removed 
where it will not cause confusion, {\em e.g.,}
the discretized variable-dependent source $f_k$ is denoted $f$. 
The solution dependent part of the residual $g(u^n)$ is denoted $g^n$, 
and residual $r^n$ is given by
\begin{equation}\label{def:residual}
r^n = -g^n + \delta f,
\end{equation}
where $\delta = 1$ for iteration \eqref{eqn:stabit}.

Before introducing the definitions of the numerical dissipation parameter $\gamma$
and the scaling parameter $\delta$, local convergence the $\sigma$-Newmark 
iteration~\eqref{eqn:stabit} is recalled from~\cite{Pollock15a}.  As the scaling
parameter $\delta$ is held constant over iterations on a given refinement, 
Theorem~\ref{thm:localconvergence} holds also for the inexact 
iteration~\eqref{eqn:inexactit}, where in the first case the residual which 
converges to zero is given by $r(u,x) = -g(u) + f(x)$ and in the second 
it is given by $r(u,x) = -g(u) + \delta f(x)$.  The proof of local convergence
depends on the nonlinear discrete problem, which coincides up to sign 
with the residual 
function on a fixed refinement, satisfying
a local Lipschitz condition on the Jacobian with respect to $u$ in the vicinity of the exact solution $u^\ast$.
For simplicity of notation, consider the convergence of $-r(u,x)$ to zero,
whose Jacobian is then $g'(u)$. Denote the open neighborhood about $u$ by 
$N(u,\eps) = \{ v  \, \rest \nr{u-v} < \eps\}$.
\begin{assumption}\label{assume:lipschitz} There exist $\omega_L,\eps > 0$ 
so that for all $w,v \in N(u^\ast, \eps)$
\[
\nr{g'(w) - g'(v) } \le \omega_L \nr{w- v} \tforall w,v \in N(u^\ast, \eps).
\]
\end{assumption}
The second set of assumptions addresses the nonsingularity, boundedness and
stability of the approximate Jacobian.  
Assumption~\ref{assume:LCsigma} requires a regularization parameter
$\alpha^n < \alpha_M$ for some fixed $\alpha_M$, 
which follows from the definition of  
$\alpha^n$ described 
in~\cite{Pollock14a}, and repeated here for convenience.  
In accordance with the local convergence Theorem~\ref{thm:localconvergence} 
the update assures $\alpha^n \le \nr{r^n}$, 
so long as the residual is decreasing.
\begin{definition}[Regularization parameter $\alpha$]\label{update:alpha} 
Set $\beta^0 = 1$. For $n \ge 1$,
\begin{align*}
\alpha^{n} &= \beta^n \nr{r^n}, \quad ~\text{ with } \\
~\beta^n &= \left\{ \begin{array}{cc}
\min\{ 1~, \max\{ \beta^{n-1}/2~, \nr{r^{n}}/\nr{r^{n-1}} \} \}, & \text{ if } 
\nr{r^{n}}< \nr{r^{n-1}},\\
\beta^{n-1}, & \text{ otherwise. }
\end{array}\right. 
\end{align*}
\end{definition}
The factor $\beta^n$ of Definition~\ref{update:alpha} 
modifies $\alpha^n$ to reduce the regularization faster
if the residual is decreasing steadily, without inducing rapid 
fluctuations in the regularization between iterations.
So long as the residual is decreasing, $\alpha^n \le \nr{r^n}$ 
should be maintained to satisfy the hypotheses of the local convergence
Theorem~\ref{thm:localconvergence}.
The last set of assumptions for the local convergence theorem
addresses the invertibility and stability 
of the approximate Jacobian.
\begin{assumption}\label{assume:LCsigma} 
There exist $\beta > 0$ so that for $0 < \sigma_0 \le \sigma \le 1$, fixed 
$\bar u$, positive semidefinite $R$, $\gamma > 1$, and for all  
$0 < \alpha_n < \alpha_M$, then for all $u \in N(u^\ast, \eps)$:
\begin{enumerateX}
\item $ \alpha_n R + \gamma\{ (1 - \sigma) g'(\bar u) - \sigma  g'(u) \}$ is invertible.
\item $\nr{(\alpha_n R + \gamma \{(1 - \sigma) g'(\bar u) - \sigma  g'(u) \})^{-1}} \le M_{\gamma\sigma}$.
\item $\nr{ (\alpha_n R)(\alpha_n R + \gamma \{ 1 - \sigma) g'(\bar u) - \sigma g'(u) \})^{-1}} \le \f{1}{1 + \beta \gamma /\alpha_n}.$
\end{enumerateX}
\end{assumption}
The local convergence theorem demonstrates that for a sufficiently small residual,
if the incoming iterate $u^n$ is within the basin of attraction of the 
exact solution $u^\ast$, then $u^n$ is also within the basin of attraction,
and the residual decreases asymptotically at the rate 
$\nr{r^{n+1}} \le (1 - 1/\gamma) \nr{r^n}$, for $\gamma > 1$.
\begin{theorem}\label{thm:localconvergence} 
Let $\alpha_n \le \nr{r(u^n)} \le \alpha_M$ and let 
Assumptions~\ref{assume:lipschitz} and~\ref{assume:LCsigma} hold, and
$\alpha^n$ given by Definiteion~\ref{update:alpha}.  
Define $\sigma$ by
\begin{align}\label{eqn:sigmadefTh}
\sigma = \max \left\{\sigma_0,\, 1 - \f{\nr{r^n}}{K_0} \right\}, 
\end{align}
for a given $K_0$.  Then there exists $\bar d >0$ such that for $u^n$ 
in the  open set $\cS$ given by
\begin{align}
\cS = \left\{ u \in N(u^\ast, \eps) \, \Rest \, \nr{r(u)} <\bar d \right\},
\end{align}
 and $u^{n+1}$ defined by iteration~\eqref{eqn:stabit}, 
or respectively by \eqref{eqn:inexactit} with $\gamma > 1$, 
it holds that $u^{n+1} \in \cS$, and the sequence of residuals converges $q$-linearly to zero with asymptotic rate $q = 1-1/\gamma$.
\end{theorem}
The proof of Theorem~\ref{thm:localconvergence} is given in~\cite{Pollock15a}.

\section{Update of numerical dissipation.}\label{sec:gamma_update}

The numerical dissipation parameter $\gamma \ge 1$ 
of iterations \eqref{eqn:stabit} and
\eqref{eqn:inexactit} should start large enough to provide
sufficient stability to the iterations and decrease down to one in order to 
recover the quadratic convergence rate of the Newon-like iterations.
The parameter is adjusted during the course of iterations on a given refinement
whenever the observed rate of convergence is within tolerance of the 
expected rate, $(1 - 1/\gamma)$.
In order to ensure a decrease in $\gamma$ as it is updated, each $\gamma^n_k$
must be bounded by an appropriate $\GM$, with respect to the user set rate
 tolerance, $\eps_T$.
\begin{assumption}\label{assume:maxgamma}Given a rate tolerance $\eps_T$, 
$\gamma_k^n$ is bounded above for all $n, ~k$ by $\gamma_{\text{MAX}}$, 
satisfying
\begin{equation}\label{as:gammamax}
\gamma_k^n \le \gamma_{\text{MAX}} < 1/\eps_T.
\end{equation}
\end{assumption}
This assumption has been observed to be necessary, and $\gamma$ can indeed
grow without bound if allowed greater than $1/\eps_T$. It is also
of practical value not to set the rate tolerance $\eps_T$ too low, 
as the update of $\gamma$ can
stall for some number of refinements due to the partial capturing of the 
solution and variable dependent layers in the preasymptotic regime. 
In the 
examples presented in Section~\ref{sec:numerics}, $\eps_T$ is chosen either
as $0.005$ or $0.01$, where the larger value is chosen in the case where 
steep gradients prevent the resolution of the Jacobian, and this adjustment
is sufficient to speed the update of $\gamma$ without compromising stability
of the method. With these parameters, the update of $\gamma$ on a given refinement $k$ is performed when the following
criteria are met.
\begin{criteria}[$\gamma$ update]\label{assume:gamma} 
Given a user set tolerance $\eps_T$, and $\gamma^n < \GM$ satisfying 
Assumption~\ref{assume:maxgamma}, and a
minimum number of iterations between updates $I_{\text{min}}$, the dissipation
parameter $\gamma^n$ is updated after each iteration $n$ on which
 all the following criteria hold.
\begin{enumerateX}
\item $\gamma^n > 1$.
\item The observed rate of convergence is within tolerance of the predicted value
\begin{equation}\label{as:absRateTol}
\left| \f{\nr{r^{n+1}}}{\nr{r^n}} -\left(  1-  \f1 {\gamma^n} \right) \right| 
< \eps_T.
\end{equation}
\item The observed rate of convergence is sufficiently stable. 
\begin{equation}\label{as:relRateTol} 
\left| \f{\nr{r^{n+1}}}{\nr{r^n}} -  \f{\nr{r^{n}}}{\nr{r^{n-1}}} \right|
< \eps_T. 
\end{equation}
\item At least $I_{\text{min}} \ge 2$ iterations have passed since the last update
of $\gamma$.
\end{enumerateX}
\end{criteria}

The last condition of Criteria~\ref{assume:gamma} is necessary to assure the three 
consecutive residuals of~\eqref{as:relRateTol} are computed with the same
value of $\gamma$, and show the requisite rate stability.
Given a reduction factor $q = q_\gamma < 1$, on satisfaction of 
Criteria~\ref{assume:gamma} after iteration $n$ of 
fixed refinement $k$, $\gamma^{n+1}$ is updated as follows.
\begin{definition}\label{def:gamma}
Given an initial $\delta_0 > 1$, for $k > 0$
\begin{align}\label{eqn:gammadef}
\tilde \gamma^{n+1} = q \cdot\f{\langle r^n,r^n\rangle}
{\langle r^n, g^{n+1} - g^n\rangle}, 
\quad \gamma^{n+1} = \max\left\{ 1~, \tilde \gamma^{n+1} \right\}.
\end{align} 
\end{definition}
To see this update is reasonable as the algorithm attains convergence, observe the ideal value of $g^{n+1}$ is $\delta f$, implying $r^{n+1} = 0$,
and the ideal value of
$g^{n+1} - g^n$ is  $\delta f - g^n = r^n$, 
so that $\gamma^{n+1}= 1$, if the iteration
produces a zero residual. Based on condition~\eqref{as:absRateTol} of 
Assumption~\ref{assume:gamma}, the updated quantity $\gamma^{n+1}$ 
is bounded above and below
in terms of $\gamma^n$, and is then shown to be decreasing, on the condition of
Assumption~\ref{assume:monogamma}.  
Lemma~\ref{lemma:gammadenom} 
bounds the denominator $\langle r^n, g^{n+1} - g^n \rangle$, then 
Corollary~\ref{cor:gammabound} bounds the ratio determining $\gamma^{n+1}$.

\begin{lemma}\label{lemma:gammadenom} Under \eqref{as:absRateTol} of Assumption~\ref{assume:gamma}, 
with tolerance $\eps_T$ the following upper and lower bounds hold.
\begin{align}\label{bound:rg}
\nr{r^n}^2\left(\f 1{\gamma^n} - \eps_T\right) 
< \langle r^n, g^{n+1} - g^n \rangle < 
\nr{r^n}^2 \left(2 - \f 1{\gamma^n} + \eps_T\right).
\end{align} 
\end{lemma}
\begin{proof}
Rewriting $g^n \an g^{n+1}$ in terms of the residual
\begin{equation}\label{eqn:boundrg_001}
\langle r^n, g^{n+1} - g^n \rangle 
= \langle r^n, -g^n + \delta f - ( -g^{n+1} + \delta f)\rangle
= \langle r^n, r^n - r^{n+1}\rangle.
\end{equation}
For the first inequality
\begin{align}\label{eqn:boundrg_002}
\langle r^n, r^n\rangle - \langle r^n, r^{n+1}\rangle
\ge \nr{r^n}^2 - \nr{r^n}\nr{r^{n+1}}  
= \nr{r^n}^2 \left( 1 - \f{\nr{r^{n+1}}}{\nr{r^n}}\right).
\end{align}
Bounding the ratio of residuals on the right by \eqref{as:absRateTol} 
\begin{align}\label{eqn:boundrg_003}
1 - \f{\nr{r^{n+1}}}{\nr{r^n}} > \f 1 {\gamma^n} - \eps_T.
\end{align}
Applying \eqref{eqn:boundrg_002} and \eqref{eqn:boundrg_003} to 
\eqref{eqn:boundrg_001}, the result follows.

For the second inequality of \eqref{bound:rg}, applying Cauchy-Schwarz and
triangle inequalities to~\eqref{eqn:boundrg_001}
\begin{align}\label{eqn:boundrg004}
\langle r^n, r^n - r^{n+1}  \rangle
\le \nr{r^n}\nr{r^{n} - r^{n+1}}
\le \nr{r^n}^2 + \nr{r^n}\nr{r^{n+1}}.
\end{align}
By \eqref{as:absRateTol}, $\nr{r^{n+1}} < \nr{r^n}(1 - 1/\gamma + \eps_T)$,
from which the result follows.
\end{proof}
The upper and lower bounds on $\tilde \gamma^{n+1}$ now follow from ~\eqref{bound:rg}.
\begin{corollary}\label{cor:gammabound}
Under Assumption~\ref{assume:maxgamma} and Criteria~\ref{assume:gamma}, 
the update of $\gamma^n$ to 
$\tilde \gamma^{n+1}$ of~\eqref{eqn:gammadef} satisfies the following bounds.
\begin{align}\label{eqn:gammatildebound}
 \f{q \cdot \gamma^n}{\gamma^n(2 + \eps_T) -1} < 
\tilde \gamma^{n+1}
< \f{q \cdot \gamma^n}{1 - \eps_T \gamma^n}.
\end{align}
\end{corollary}
\begin{proof}
The argument follows from direct application of Lemma~\ref{lemma:gammadenom}
to the definition of $\tilde \gamma^{n+1}$ given by~\eqref{eqn:gammadef}. 
The upper bound
\begin{align}\nonumber
\tilde \gamma^{n+1} = q \cdot \f{\nr{r^n}^2}{\langle r^n, g^{n+1} - g^n \rangle} 
< q \cdot \f {1}{1/\gamma^n - \eps_T} = \f{q \cdot \gamma^n}{1 - \eps_T \gamma^n},
\end{align}
holds under Assumption~\ref{assume:maxgamma}, which assures the positivity of the 
denominator. Similarly, the lower bound follows applying $\gamma > 1$, 
the first condition
of update Criteria~\ref{assume:gamma}.
\end{proof}
The lower bound remains satisfied in the case $\gamma^{n+1} = 1$ by 
Definition~\ref{def:gamma},
 and the upper bound is of no interest as the goal of the sequence is to 
decrease $\gamma$ to unity in a stable manner. To establish the monotonic decrease
of $\gamma$, and more generally the decrease at rate $1>\bar q > q$, the following
Assumption~\ref{assume:monogamma} strengthens the Assumption~\ref{assume:maxgamma},
used to establish the upper bound of~\eqref{eqn:gammatildebound}.
\begin{assumption}[Monotone decrease of $\gamma$]\label{assume:monogamma}
Given a reduction factor $q$, a rate of decrease $\bar q > q$,
  and rate tolerance $\eps_T$, let $\gamma_k^n$ be
bounded above for all $n, ~k$ by $\gamma_{\text{MONO}}$ satisfying
\begin{equation}\label{eqn:gammamono}
\gamma_k^n \le \gmo = \f {1}{\eps_T}\left( 1 - \f{q}{\bar q} \right).
\end{equation}
\end{assumption}
In practice, setting $\bar q= 1$ and $\gamma < (1-q)/\eps_T$ ensures 
$\gamma$ will be nonincreasing, although for more stability in the
earlier iterations, the weaker condition $\gamma < \GM$ is sufficient,
and after $\gamma$ reduces below $\gmo$ the monotonicity holds. 
\begin{corollary}[Monotonic decrease of $\gamma$]\label{cor:monodecrease}
For $\gamma^{n+1}$ determined by ~\eqref{eqn:gammadef} on the satisfaction
of Criteria~\ref{assume:gamma}, and Assumption~\ref{assume:monogamma} with $\bar q = 1$, the sequence $\{\gamma^n\}_{n\ge 0}$ is nonincreasing. 
Moreover, assuming~\eqref{eqn:gammamono} with
$1 < \bar q < q$, then for $\gamma^{n+1} = \tilde \gamma^{n+1}> 1$, 
the sequence satisfies $\gamma^{n+1}< \bar q \gamma^n$. 
\end{corollary}
The proof follows from the upper bound on $\tilde \gamma^{n+1}$, given by 
\eqref{eqn:gammatildebound} of Corollary~\ref{cor:gammabound}. 
\begin{proof}
To assure $\tilde \gamma^{n+1} \le q\gamma^n/(1 - \eps_T \gamma^n)< \bar q \gamma^n$, require
\begin{align}\label{eqn:monodecr001}
q < \bar q(1 - \eps_T \gamma^n),
\end{align} 
which is satisfied by Assumption~\ref{assume:monogamma}.
To establish the monotonic decrease without a rate, 
set $\bar q = 1$ in~\eqref{eqn:monodecr001}, and rearrange terms to obtain
$\gamma^n < (1 - q)/\eps_T$ as the necessary condition.
\end{proof}
While the determination of the monotonic decrease of $\gamma$ 
with a given rate $\bar q$ to unity follows in a straightforward manner 
from Definition~\ref{def:gamma}, condition~\eqref{as:absRateTol}, 
and Assumption~\ref{assume:monogamma}, this establishes the important property
that if the algorithm does not reset after a given point, $\gamma^N = 1$ will 
be achieved after at most $N - n$ = $\lceil -\log(\gamma^n)/\log(\bar q)\rceil$ updates, where
$\lceil \cdot \rceil$ denotes the ceiling function.
Notably, $N-n$ counts only the updates of $\gamma$, independent of the number
of mesh refinements that occur over the course of those updates.
The more delicate point of whether those updates occur, which requires 
the predicted convergence rate is achieved within tolerance, is 
addressed in Theorem~\ref{thm:inexactconvergence}, which relies on
Theorem~\ref{thm:localconvergence}, the local convergence with rate of iterations
~\eqref{eqn:stabit} and~\eqref{eqn:inexactit}.  
\section{Update of scaling parameter}\label{sec:delta_update}
The scaling parameter $\delta$ for the inexact iteration~\eqref{eqn:inexactit} 
is computed  on the completion of the iterations 
for refinement $k$.  This is in contrast to the update on $\gamma$, performed 
whenever the iterations converge close enough to the predicted rate. 
This strategy is reasonable
as updating $\delta$ changes the discrete problem which already changes on each 
refinement. 
Except for the first, the criteria to update $\delta$ 
are a subset of the conditions used to exit the Newton-like iterations on each refinement.
The last two conditions of Criteria~\ref{criteria:delta} are closely
related to~\eqref{as:absRateTol} and~\eqref{as:relRateTol}, the second and 
third conditions of Criteria~\ref{assume:gamma}, the update on $\gamma$. 
The acceptable rate criterium ~\eqref{as:accRateTol} 
is weaker than the requirement~\eqref{as:absRateTol},
and the stability criterium~\eqref{as:exitStability} requires the observed rate
has not decreased by much compared to the previous observed rate.     
Distinct from the conditions to update $\gamma$, the second condition of 
Criteria~\ref{criteria:delta} 
requires the current and previous residual to be smaller than the initial residual
of the current refinement and final residual of the previous refinement.
This condition is important to ensure the algorithm convergences, and the 
sequence of residuals $r_k$ decrease at a rate
\[
\nr{r_k} < \nr{r_{k-1}}(1 - 1/(2\gamma^n_k)) \le \nr{r_{k-1}}(1 - 1/(2\GM)),
\]
on Assumption~\ref{assume:maxgamma}, or 
\[
\nr{r_k} < \nr{r_{k-1}}(1 - 1/(2 \gamma^n_k)) \le \nr{r_{k-1}}(1 - 1/(2\gmo)),
\]
on Assumption~\ref{assume:monogamma},
where $r_k$ is the final residual on refinement $k$. 
Conditions~\eqref{as:accRateTol} and ~\eqref{as:exitStability} are also
important
to the stability of the update, and assure $g^{n+1} = g(u^n + w^n)$ has been 
improved by two consecutive steps $w^{n-1}$ and $w^n$. 
\begin{criteria}[$\delta$ update]\label{criteria:delta} 
Given a user set rate tolerance $\eps_T$ as in Criteria~\ref{assume:gamma} and
a convergence tolerance $\econ$,
the iterations are terminated 
and $\delta$ is updated after iteration $n$ of refinement $k$,
when conditions (1-4) are all met, or when conditions (1) and (5) are satisfied.  
\begin{enumerateX}
\item $\delta_k < 1$.  If $\delta_k = 1$, it is not updated.
\item Sufficient decrease in the sequence of residuals $\nr{r_k}$
\begin{equation}\label{as:delta1}
\nr{r_k^n} < \min\{ \nr{r_k^0}, \nr{r_{k-1}}  \}. 
\end{equation}
\item The observed rate of convergence is beneath the 
accepted rate $(1 - 1/2\gamma)$
\begin{equation}\label{as:accRateTol}
\f{\nr{r^{n+1}}}{\nr{r^n}} < 1 -  \f 1 {2 \gamma^n} .
\end{equation}
\item The observed rate is sufficiently stable
\begin{equation}\label{as:exitStability}
\f{\nr{r^{n+1}}}{\nr{r^n}} + \f {\eps_T} {2} > \f{\nr{r^n}}{\nr{r^{n-1}}}.
\end{equation}
\item $\nr{r^n}< \econ$. The iteration reached convergence.
\end{enumerateX}
\end{criteria}

On satisfaction of Criteria~\ref{criteria:delta}, $\delta$ is updated  based on replacing the satisfied linearized equation
\begin{align}\label{eqn:lin_relation}
\{\alpha^n R_k + \gamma^n\left( \sigma^n g'(u^n) + (1 - \sigma^n)g'(\bar u) \right) \} w^n = -g^n + \delta_k f_k,  
\end{align} 
with the corresponding relation found by replacing $g'(u^n)w^n$, with 
$g^{n+1} - g^n$, modifying the right hand side residual by $q< 1$ 
to obtain
\begin{align}\label{eqn:nonlin_relation}
\{\alpha^n R_k + \gamma^n\left((1 - \sigma^n)g'(\bar u) \right) \} w^n + 
\gamma^n \sigma^n \left(g^{n+1} - g^n \right)=
 -g^n + q \delta_{k+1} f_k,  
\end{align} 
and considering the discrete inner product of both sides against the source 
$f = f_k(x)$. 
Given a factor $q = q_\delta < 1$, at the end of each nonlinear
solve, $\delta_{k+1}$ is determined by
\begin{align}\label{eqn:delta_satisfies}
\langle f,  \{\alpha^n R_k + \gamma^n\left((1 - \sigma^n)g'(\bar u) \right) \} w^n + 
\gamma^n \sigma^n \left(g^{n+1} - g^n \right) \rangle=
\langle f,  -g^n + q_k\delta^{k+1} f \rangle,  
\end{align}
resulting in the following definition.
\begin{definition}\label{def:deltaclosed} 
Given an initial $\delta_0 < 1$, for $k > 0$
\begin{align}\label{eqn:deltaclosed}
\tilde \delta_{k+1} &= \f{\langle f, \gamma \sigma (g^{n+1} - g^n) 
+ (\gamma(1-\sigma) g'(\bar u) 
+ \alpha R) w^n + g^n \rangle}
{q_k {\nr f}^2}, \\
\delta_{k+1} & = \min\{ 1~, \tilde \delta_{k+1}\},
\end{align}
with $R = R_k, ~\alpha = \alpha^n, ~ \gamma = \gamma^n \an \sigma = \sigma^n$.
\end{definition}
The sequence for $\delta_k<1$ can also be expressed recursively by considering
\eqref{eqn:lin_relation}, satisfied exactly with  $\delta = \delta_k$, and
subtracting the corresponding inner product against $f$ from 
\eqref{eqn:delta_satisfies}.  Expressed in terms of the one-step linearization 
error
$\cL^e(u^n)$ given by 
\begin{equation}\label{eqn:lin_error}
\cL^e(u^n) \coloneqq g^{n+1} - g^n - g'(u^n) w^n, \quad g^{n+1} = g(u^n + w^n),
\end{equation}
the recursive relation for $\delta$ may be written
\begin{align}\label{eqn:delta_recursive}
\delta_{k+1} = \f 1 {q_k} \left\{ \delta_k + \f{\gamma \sigma}{ \nr{f}^2} 
\langle f, \cL^e(u^n) \rangle
\right\}.
\end{align}
From~\eqref{eqn:delta_recursive},
the local Lipschitz Assumption~\ref{assume:lipschitz} on $g'(u_k^n)$
is used to show the convergence
of $\delta$ to one as the residual attains convergence. First, a bound 
on $\delta_{k+1}$ with respect to $\delta_k$ is established.
\begin{lemma}\label{lemma:deltabound}
Assume $u_k^n \in N(u_k^\ast, \eps_k)$ of the local Lipschitz condition 
Assumption~\ref{assume:lipschitz} with respect to the approximate Jacobian
$g'(u_k^n)$, and where $u_k^\ast$ satisfies $r_k(u_k^\ast)$ = 0. 
Let $\omega_{L,k}$ denote the corresponding Lipschitz constant. 
Then $\delta_{k+1}$
updated by 
~\eqref{eqn:deltaclosed} with $q_k < 1$,
satisfies the following bound with respect to $\delta_k$.
\begin{align}\label{eqn:delta_bound}
\f 1 {q_k} \left(\delta_k - \f{\gamma \sigma \omega_{L,k}}{2 \nr f}\nr{w^n}^2 
\right)
\le \delta_{k+1} \le 
\f 1 {q_k} \left(\delta_k + \f{\gamma \sigma \omega_{L,k}}{2 \nr f} \nr{w^n}^2 
\right).
\end{align}
\end{lemma}
\begin{proof}
Writing $\delta_{k+1}$ by the recursive definition given by
~\eqref{eqn:delta_recursive}, writing out $\cL^e(u^n)$ explicitly
\begin{align}\label{eqn:deltab001}
\delta_{k+1} = \f 1 {q_k} \left\{ \delta_k + \f{\gamma \sigma}{ \nr{f}^2} 
\langle f, g(u^{n+1}) - g(u^n) - g'(u^n) w^n \rangle
\right\},
\end{align}
with $u^{n+1} = u^n + w^n$, yielding from the integral mean value theorem
\begin{align}\label{eqn:deltab002}
g(u^{n+1}) - g(u^n) - g'(u^n) w^n = \int_0^1 \left( g'(u^n + tw^n) - g'(u^n)
\right)w^n \, dt. 
\end{align}
Applying a Cauchy-Schwarz inequality to \eqref{eqn:deltab002}, and the 
Lipschitz Assumption~\ref{assume:lipschitz}
\begin{align}\label{eqn:deltab003}
\left\langle f, \int_0^1 \left( g'(u^n + tw^n) - g'(u^n) \right) w^n\, dt 
\right\rangle 
& \le \nr{f} \f{\omega_{L,k}}{2}\nr{w^n}^2.
\end{align}
Putting together~\eqref{eqn:deltab002} and~\eqref{eqn:deltab003} into
\eqref{eqn:deltab001} yields the first result, ~\eqref{eqn:delta_bound}.
This shows where the Lipschitz estimate on Jacobian holds, 
the variance in $\delta_{k+1}$ from a strict reduction by factor $q_k$ is 
bounded with respect to the Lipschitz constant $\omega_{L,k}$.
\end{proof}

The next Corollary~\ref{cor:deltacon} establishes the monotonicity of 
the sequence $\{\delta_k\}_{k \ge K_0}$ for some finite $K_0$, assuming 
the sequence of residuals $\nr{r_k}$ is decreasing. This is imposed by assuming
$\delta$ is updated on the satisfaction of Criteria~\ref{criteria:delta}, 
meaning the iterations converge within the accepted rate on each refinement.
An easily imposed lower bound for $\delta$ is assumed.
\begin{assumption}\label{assume:deltamin} There is a constant $0 < \dmin < 1$ 
for which 
$0 < \dmin \le \delta_k$ for all $k \ge 0$. 
\end{assumption}
This assumption can be imposed in the definition of 
$\delta_{k+1}$ from $\delta_k$ but is given separately for 
clarity of presentation. The lower bound $\dmin$ is related to generally 
inaccessible constants which characterize the approximate discrete problems
in the next Assumption~\ref{assume:problemcapture}.  As a guideline,
$\dmin = 1/\GM$ is reasonable as $\delta$ tends to increase roughly as
$1/\gamma$.  A need to impose $\delta > \dmin$ was not encountered
in the numerical experiments presented in Section~\ref{sec:numerics}.
Along with~\ref{assume:deltamin}, a condition on the data of the 
approximate problems is assumed,
requiring the ratio of $\omega_{L,k}$,
the Lipschitz constant for the approximate problem on refinement $k$,
to $\nr{f_k}$,  
remains bounded.
\begin{assumption}\label{assume:problemcapture} There is a constant $B_{\omega f}$ 
for which after some refinement $K_1$
\begin{equation} 
\f {\omega_{L,k}} {\nr{f_k}} \le B_{\omega f}
~\tforall k \ge K_1, 
\end{equation}
and the convergence tolerance $\econ$ of Criteria~\ref{criteria:delta} satisfies
\begin{equation}
\econ^2\le \f{\dmin (1 - q_\delta/\bar q)}{\gamma_k^n \sigma_k^n}
\f{2}{ M^2_{\gamma \sigma}}B_{\omega f}^{-1}, ~\tforall k \ge K_1.
\end{equation}
\end{assumption}
\begin{corollary}[Convergence of $\delta$ to unity]\label{cor:deltacon}
Let Assumption~\ref{assume:LCsigma} and
the hypotheses of Lemma~\ref{lemma:deltabound} hold.
Then $\delta_{k+1}$ updated on satisfaction of Criteria~\ref{criteria:delta} with
$q_k < q_\delta < 1$ satisfies $\delta_{k+1} > \delta_k/\bar q$ \ for
$q_\delta < \bar q < 1$ whenever  
\begin{equation}\label{delta:mono}
\delta_{k} > \f{\gamma^n \sigma^n}{(1 - q_\delta/\bar q)} 
\f{\nr{r^n}^2 M_{\gamma \sigma}^2}{2}\f{\omega_{L,k}}{\nr {f_k}}.
\end{equation}
Additionally, if Assumptions~\ref{assume:deltamin} and~\ref{assume:problemcapture}
 hold, then~
\eqref{delta:mono} holds for all refinements $k > K_0> K_1$, so long as 
the residual
$\nr{r_{K_0}}$ satisfies 
\begin{equation}\label{delta:resimono}
\nr{r_{K_0}}^2\le \f{\dmin (1 - q_\delta/\bar q)}{\gamma^n \sigma^n}
\f{2}{ M^2_{\gamma \sigma}}B_{\omega f}^{-1},
\end{equation}
and $\delta_k = 1$ for all $k > K$ for some $K \le  \lceil  \log(\delta_{K_0})/
\log{\bar q} \rceil + K_0$, where $\lceil \cdot \rceil$ denotes the ceiling function.
\end{corollary}
\begin{proof}
Denoting the approximate Jacobian on iteration $n$ by
\[
M = \left(\alpha^n R  + \gamma^n \left( \sigma g'(u^n) + (1 - \sigma) g'(\bar u) \right) \right).
\]
In accordance with Assumption~\ref{assume:LCsigma}, 
$\nr{M^{-1}} \le M_{\gamma \sigma}$, and $\nr{w^n} \le \nr{r^n} M_{\gamma \sigma}$.
Then applying $q_\delta > q_k$
\begin{align}\label{eqn:deltamono001}
\f{\gamma^n \sigma^n}{(1 - q_\delta/\bar q)} 
\f{\nr{r^n}^2 M_{\gamma \sigma}^2}{2}\f{\omega_{L,k}}{\nr {f_k}}
 \ge \f{\gamma^n \sigma^n}{(1 - q_k/\bar q)} 
\f{\nr{w^n}^2}{2}\f{\omega_{L,k}}{\nr {f_k}}.
\end{align}
Satisfaction of~\eqref{delta:mono} then implies $\delta_k > \delta/\bar q$
when $\delta_k$ is greater than the quantity on the right of 
~\eqref{eqn:deltamono001}.
Rearranging terms in~\eqref{eqn:deltamono001} and comparison with
~\eqref{eqn:delta_bound} yields the first result
\begin{align}\label{eqn:deltamono002}
\delta_{k+1} \ge
\f 1 {q_k} \left(\delta_k - \f{\gamma \sigma \omega_{L,k}}{2 \nr f}\nr{w^n}^2 
\right)
\ge \f 1 {\bar q} \delta_k.
\end{align}
Similarly for the second result, $B_{\omega f}^{-1} < \nr{f_k}\omega_{L,k}$, 
$~\dmin < \delta_k$, and for $k \ge K_0$, Criteria~\ref{criteria:delta}, 
conditions either (2) or (4) and Assumption~\ref{assume:problemcapture} yield
$\nr{r_k} \le \nr{r_{K_0}}$,
 by which satisfaction of~\eqref{delta:resimono} implies
\begin{align}\label{eqn:deltamono003}
\nr{r_k}^2\le \f{\delta_k (1 - q_\delta/\bar q)}{\gamma^n \sigma^n}
\f{2}{ M^2_{\gamma \sigma}}\f{\nr{f_k}}{\omega_{K,L}}.
\end{align}
Rearranging terms again shows~\eqref{delta:mono}. So far, this shows
$\delta_{k+1} > \delta_k/\bar q$ for $k \ge K_0$.  A short calculation
and Definition~\ref{def:deltaclosed}
yield the final result, $\delta_k = 1$ after at most an additional
$\lceil  \log(\delta_{K_0})/\log{\bar q} \rceil$ refinements and
updates of $\delta$. 
\end{proof}

Importantly, the recursive definition 
\eqref{eqn:delta_recursive} shows $\delta$ is increased more aggressively
when the linearization error lies in the direction of $f$, and is increased
less and possibly decreased when $\cL^e(u^n)$ is in opposition to $f$.  
This sensitivity to direction as opposed to magnitude alone yields added 
stability to the method. More precisely, the following proposition characterizes
the conditions supporting increase of $\delta$ in terms of the projection of
the linearization error along $f$. The normalized vector in the direction of 
$f$ is denoted $\hat f = (1/\nr f) f$.
\begin{proposition}[Characterization of increase in $\delta$]
The sequence $\{ \delta_k\}$ given by the recursive definition
~\eqref{eqn:delta_recursive} satisfies $\delta_{k+1} > (1/ \bar q) \delta_k, ~ 1 < \bar q \le q_k$ whenever 
\begin{align}\label{eqn:del_increase}
\langle \hat f, \cL^e(u^n)\rangle > -{\delta_k}\left( 1 - \f{q_k}{\bar q}\right)
\f{\nr f}{\gamma^n_k \sigma^n_k}.
\end{align}
\end{proposition} 
The criteria~\eqref{eqn:del_increase} follows by direct calculation from
\eqref{eqn:delta_recursive}. The lower bound on~\eqref{eqn:delta_bound} ensures
the increase in $\delta_k$ in the asymptotic regime as $\nr{w^n} \goto 0$ on 
a given refinement assuring the algorithm eventually convergences to the given 
unscaled data, $\delta = 1$; however, in practice the increase in $\delta_k$ is governed by ~\eqref{eqn:del_increase} through the preasymptotic and coarse mesh regimes.

\section{Algorithm}\label{sec:algorithm}
The results of the previous sections are summarized in the following algorithm,
and combined to yield a proof of convergence using iteration~\eqref{eqn:inexactit}.
The solver using inexact iteration~\eqref{eqn:inexactit} is implemented 
in an adaptive method according to the following basic algorithm.
The algorithm using ~\eqref{eqn:stabit} is the same, except $\delta = 1$ 
is assumed throughout. 

The additional criteria and parameter settings necessary to define the method are
summarized below the algorithm. The choice of parameters to
guarantee the monotonic decrease of $\gamma$ as in Corollary~\ref{cor:monodecrease}
is listed as optional. The selection of 
parameters $\dmin$ and $\econ$ to guarantee the monotonic 
increase in $\delta$ as in
Corollary~\ref{cor:deltacon} depend on 
unknown constants that are not assumed available. 
\begin{algorithm}[Algorithm using the inexact iteration~\eqref{eqn:inexactit}]
\label{algorithm:inexact}
Set the parameters $q_\gamma, ~\eps_T$ and $\GM$ in accordance with
Assumption~\ref{assume:maxgamma} and optionally~\ref{assume:monogamma}. 
Set the convergence tolerance $\econ$ and  $\dmin$ of 
Assumption~\ref{assume:deltamin}. Let $\bar u = 0$.
Start with initial $u^0$, $\gamma^0$, $\delta_0$. 
On partition $\cT_k, ~ k = 0, 1, 2, \ldots$
\begin{enumerateX}
\item Compute $R_k$ by Definition~\ref{def:Rk}, and $g'(\bar u)$.
\item Set $r^0 = -g(u^0) + \delta_0 f_0$, and set $\alpha_0 = \nr{r^0}$.
\item While Exit criteria~\ref{criteria:exit} are not met on iteration $n-1:$
\begin{enumerateY}
  \item Set $\sigma$ according to~\eqref{eqn:sigmadefTh}.
  \item Solve $(\alpha^n R_k + \gamma^n\{ \sigma^n g'(u^n) + 
               (1-\sigma)g'(\bar u) \}) 
               w^n = r^n,$ \ for $~w^n$.
  \item Update $u^{n+1} = u^n + w^n$, and $r^{n+1} = -g(u^{n+1}) + \delta_k f_k$.
  \item If Criteria~\ref{assume:gamma} are satisfied, update $\gamma^{n+1}$ 
       according to Definition~\ref{def:gamma}.
  \item Update $\alpha^n$ according to ~\eqref{update:alpha},
\end{enumerateY}
  \item If Criteria~\ref{criteria:delta} are satisfied, 
        update $\delta_{k+1}$ for partition $\cT_{k+1}$ according 
        to~\eqref{def:deltaclosed} with $q_\delta = (q_\gamma)^{P+1}$,
        where $P$ counts the number of times $\gamma$ was updated
        on refinement $k$.
  \item Compute the error indicators to determine the next
        mesh refinement.
\end{enumerateX}
\end{algorithm}

\begin{criteria}[Exit criteria]\label{criteria:exit}  Given a user set tolerance 
$\econ$, and an accepted rate of convergence given by $\qacc = 1-1/(2 \gamma)$
as in Criteria~\ref{criteria:delta}, and a baseline number of allowed iterations
$\ibase$, set the maximum number of iterations
$\imax$ by
\begin{equation}\label{set:imax}
 \iacc = \left\lceil \f{ \log\left( \nr{r_{k-1}}\right) - 
       \log  \left( \nr{r_k^0} \right) }{\log(\qacc)} \right\rceil + 1,
\quad \imax = \max\{\iacc, \ibase\}.
\end{equation}
Exit the solver on partition $\cT_k$ after calculating residual $\nr{r^{n+1}}$ 
when one of the following conditions holds.
\begin{enumerateX}
\item Condition (5) of Criteria~\ref{criteria:delta}: 
      the iteration reached convergence. 
\item Conditions (2-4) of Criteria~\ref{criteria:delta}:
      the iteration converges an at acceptable rate and is stopped before
      convergence. 
\item Maximum number of iterations $\imax$ completed: reset $\gamma_{k+1}^0,
      \delta_{k+1}, \an u_{k+1}^0$.
\end{enumerateX}
\end{criteria}
\begin{remark} Setting the maximum number of iterations by~\ref{set:imax}
       prevents a premature exit when the residual is converging close
       to the predicted rate.  
       An additional failure criteria may be added to exit the 
       iterations quickly if the convergence rate drifts substantially 
       higher that the accepted rate.  However, this is mainly useful on
       very coarse meshes where the problem is small and the linear solves
       take negligible time.
\end{remark}

\begin{remark}[Initial and reset values for $\gamma, ~\delta, ~u$] 
\label{remark:initialreset}
An initial set of values for $\gamma, ~\delta, \an u$ are not formally a 
feature of Algorithm~\ref{algorithm:inexact}, however 
some set of values is required to 
start the algorithm.  The examples shown in Section~\ref{sec:numerics} all use
\begin{equation}\label{eqn:initialset}
u_0^0 = 0, \quad \gamma^0_0 = \f{\nr{f_0}}{\nr{g'(u_0^0)}}, \quad \delta_0 = 1/\gamma^0_0.
\end{equation} 
If the iterations fail to converge ,{\em i.e.,} the iterations are terminated
on iteration $k$
by condition (3) of Criteria~\ref{criteria:exit}, $u$ is reset and
$\gamma \an \delta$ are updated from their previous values by
\begin{equation}\label{eqn:reset}
u_{k+1}^0 = 0, \quad \gamma_{k+1}^0 = \min\{2\gamma_k, \GM\}, 
\quad \delta_{k+1} = \max\{\delta_k/2, \dmin\}.
\end{equation}
Alternatively, $\gmo$ can take the place of $\GM$ in~\eqref{eqn:reset}.
\end{remark}
Allowing the possibility of reset 
is considered an essential part of the algorithm, as the 
method is designed for capturing steep layers starting
from a coarse mesh where the layers may be only partially uncovered 
or even completely 
undetected.  The approximate discrete problem can drastically change
over certain mesh refinements, and the latest update of parameters based on 
a milder approximate problem may not provide sufficient stability.

The next Assumption~\ref{assume:finitereset} supposes 
Algorithm~\ref{algorithm:inexact} resets only finitely-many times,
{\em e.g.,} after some iteration $K_A$ the sequence of problems enters the 
preasymptotic regime and iterations on subsequent refinements are terminated
by either the first or second condition of Criteria~\ref{criteria:exit}.
\begin{assumption}\label{assume:finitereset}
There exists a refinement $K_A$, for which 
either condition (1) or condition (2) of Criteria~\ref{criteria:exit}
is used to terminate the iterations on refinement $k$, for all $k \ge K_A$.
\end{assumption}

Regardless of whether the exact problem within the target class is well
posed, the approximate problems may not be.
The coarse mesh problems
are often ill-posed, may feature indefinite and ill-conditioned
Jacobians, and there is no 
guarantee that even an exact solution to a given approximate problem will
land in any basin of attraction for the approximate problem on the 
subsequent refinement.  

An analysis of the problem class where Algorithm~\ref{algorithm:inexact} can
be proven to converge, {\em i.e.,} where Assumption~\ref{assume:finitereset}
can be shown to hold, is beyond the scope of the current paper.  
The numerical examples in the next section demonstrate the performance of the
algorithm on a range of problems with steep internal and boundary layers, as
well as highly variable solution dependent frequencies in the diffusion 
coefficient. The following Theorem~\ref{thm:inexactconvergence} summarizes
the theoretical convergence properties of Algorithm~\ref{algorithm:inexact}
within the problem class characterized by Assumption~\ref{assume:finitereset}.

\begin{theorem}\label{thm:inexactconvergence}
Assume the hypothesis of 
Corollary~\ref{cor:monodecrease} and Corollary~\ref{cor:deltacon}, and let
Assumption~\ref{assume:finitereset} hold for $k \ge K_A$. 
Then the iterations on refinements $k \ge K_B$ for 
$K_B \le K_A + N_A + N_B$ with $N_A$ and $N_B$ given by~\eqref{eqn:con001}
and~\eqref{eqn:con002}, 
attain convergence up to tolerance $\econ$ with $\delta = 1$.  
Further assuming the hypotheses of
Theorem~\ref{thm:localconvergence}, there is 
a refinement $K_C$ for which $\gamma_k = 1$ for $k \ge K_C$. 
\end{theorem}
The convergence of $\gamma$ to one is not necessary for the convergence of the
iterations, but it is necessary for its efficiency, and the asymptotically
quadratic convergence of the resulting iterations as discussed in
~\cite{BaRo80a,KeKe98,Pollock14a,Pollock15a}.  
\begin{proof}
The convergence of the sequence of residuals $\nr{r_k} \goto 0$ within 
tolerance $\econ$ follows directly from Criteria~\ref{criteria:delta}.  
Assuming $\nr{r_{K_A}} > \econ$ and $\gamma \le \gmo$ from 
Assumption~\ref{assume:monogamma}, direct calculation shows after at most $N_A$
refinements, the iterations achieve convergence of the residual, 
with $N_A$ given by
\begin{equation}\label{eqn:con001}
N_A = \left\lceil \f{\log(\econ/\nr{r_{K_0}})}{\log(1 - 1/2\gmo)}  \right\rceil.
\end{equation}
In particular, $\nr{r_k} < \econ$ for $k \ge K_B = K_A + N_A$. From 
Corollary~\ref{cor:deltacon} and Assumption~\ref{assume:problemcapture}, 
it also holds if $\delta_{K_B}< 1$, the sequence $\delta_k, k \ge K_B$ 
increases monotonically to unity, and for $q_\gamma < \bar q < 1$ it follows that 
$\delta_k = 1$ 
for $k \ge K_C = K_B + N_B$, with 
\begin{align}\label{eqn:con002}
N_B = \lceil \log{\delta_{K_B}/\log \bar q} \rceil.
\end{align}
This establishes the first result, the convergence of the sequence of residuals
to zero for the problem defined with consistent data.

The second part of the argument  demonstrates the decrease of $\gamma$ 
to $1$ follows 
from the local convergence Theorem~\ref{thm:localconvergence}. Assuming the 
residual is small enough, then the rate $\nr{r^{n+1}}/\nr{r^n}$
asymptotically approaches $1- 1/\gamma$, and in particular is within
$\eps_T$ tolerance of the predicted rate, meaning Criteria~\ref{assume:gamma}
hold, and $\gamma$ is updated by Definition~\ref{def:gamma}; and, by
Corollary~\ref{cor:monodecrease} converges to one after at most
an additional $N_C = \lceil -\log(\gmo)/\log \bar q \rceil$ updates.
The subtle point is the convergence tolerance $\econ$ must be small enough
for the asymptotic convergence rate of Theorem~\ref{thm:localconvergence}
to be assured.
\end{proof}
\begin{remark}It is not just an artifact of the analysis that $\gamma \goto 1$
cannot be assured until after the iterations already achieve convergence to 
tolerance on every iteration, and that the tolerance $\econ$ must be 
sufficiently small.  The phenomenon of $\gamma$ stalling above one and 
failing to update for several iterations after $\nr{r_k} < \econ$ has been 
observed, in particular in problems similar to Example~\ref{example2}
in the next section, where a diffusion coefficient with a high frequency
dependence on the solution fails to resolve until the mesh in the 
highest-frequency region is sufficiently fine.  It is more commonly
observed that the iterations do converge at the predicted rate when the 
residual is far from convergence, but generally $\gamma \goto 1$ within 
a few iterations on either side of the residual converging to tolerance.
\end{remark}

\section{Numerical Examples}\label{sec:numerics}
The numerical experiments were performed with a Python implementation
of the FEniCS library~\cite{LMW12a}, 
and additional Numpy~\cite{JOP01a} computations.
For each of the examples shown, Algorithm~\ref{algorithm:inexact} is demonstrated 
solving for both a known and an unknown solution.  The unknown solution examples
demonstrate the effectiveness of the method for cases where the variable dependent
and solution dependent layers may not match up. 
For the source functions designed to produce a known solution,
the $H^1$ and $L_2$ errors denoted by $|u - u_k|_1$ and
$|u - u_k|_0$, respectively,
are shown along with the error estimator; otherwise, the error estimator is 
shown to decay at the expected rate
of $n^{-1/2}$ for degrees of freedom $n$, after the resolution of the layers. 

The examples are all run with reduction factor $q_\gamma = 0.9$ and 
$\sigma_0 = 0.9$.  The D\"orfler parameter of~\eqref{eqn:dorflermark} 
is set to $\theta = 0.2$, at the tolerance $\econ = 10^{-7}$. 
Example~\ref{example2} with the unknown source~\eqref{ex2:source2} uses 
rate tolerance $\eps_T = 0.01$, and all other examples use $\eps_T = 0.005$.
The higher rate tolerance is used to prevent the update of $\gamma$ from 
stalling in cases where a high mesh resolution is necessary to accurately
capture the Jacobian.  Perturbing these parameter settings may change 
the number of refinements or the number of resets taken for the 
algorithm to arrive at the asymptotic regime, but not the end result.
The initial and reset values of $u, \gamma$, and $\delta$ are set as in
Remark~\ref{remark:initialreset} in accordance with 
Assumption~\ref{assume:gamma} but not necessarily 
Assumption~\ref{assume:monogamma}. Nonetheless, monotonicity of the parameters
is generally observed on iterations when the algorithm does not reset.
Each example is started on a uniform mesh of 144 triangular elements.
\begin{example}[Steep internal layer]\label{example1} Consider the quasilinear 
diffusion equation on $\Omega = [0,1] \times [0,1]$.
\begin{align}\label{eqn:ex1}
F(u,x) \coloneqq - \div(\kappa(u) \grad u) - f(x,y) = 0 \inn \Omega, 
~ u = 0 \on \pa \Omega.
\end{align}
The solution dependent diffusion coefficient is given by
\begin{align}
\kappa(s) = k + \f{1}{\eps + (s-a)^2}, \quad \with~ a = 0.5, ~k = 1, ~\an 
~\eps = 6 \times 10^{-5}.
\end{align}
The following source functions are considered
\begin{align}\label{ex1:source1}
f(x,y)& \text{ chosen so }  u(x,y) = \sin (\pi x) \sin (\pi y), 
\\ \label{ex1:source2}
f(x,y)& = 10^5 \cdot \left( 0.5 - \left|x - 0.5 \right| \right)^2
\left(0.5 - \left| y - 0.5 \right| \right).
\end{align}
\end{example}
\begin{figure}
\centering
\includegraphics[trim=90pt 145pt 90pt 145pt, 
clip=true,width=0.4\textwidth]{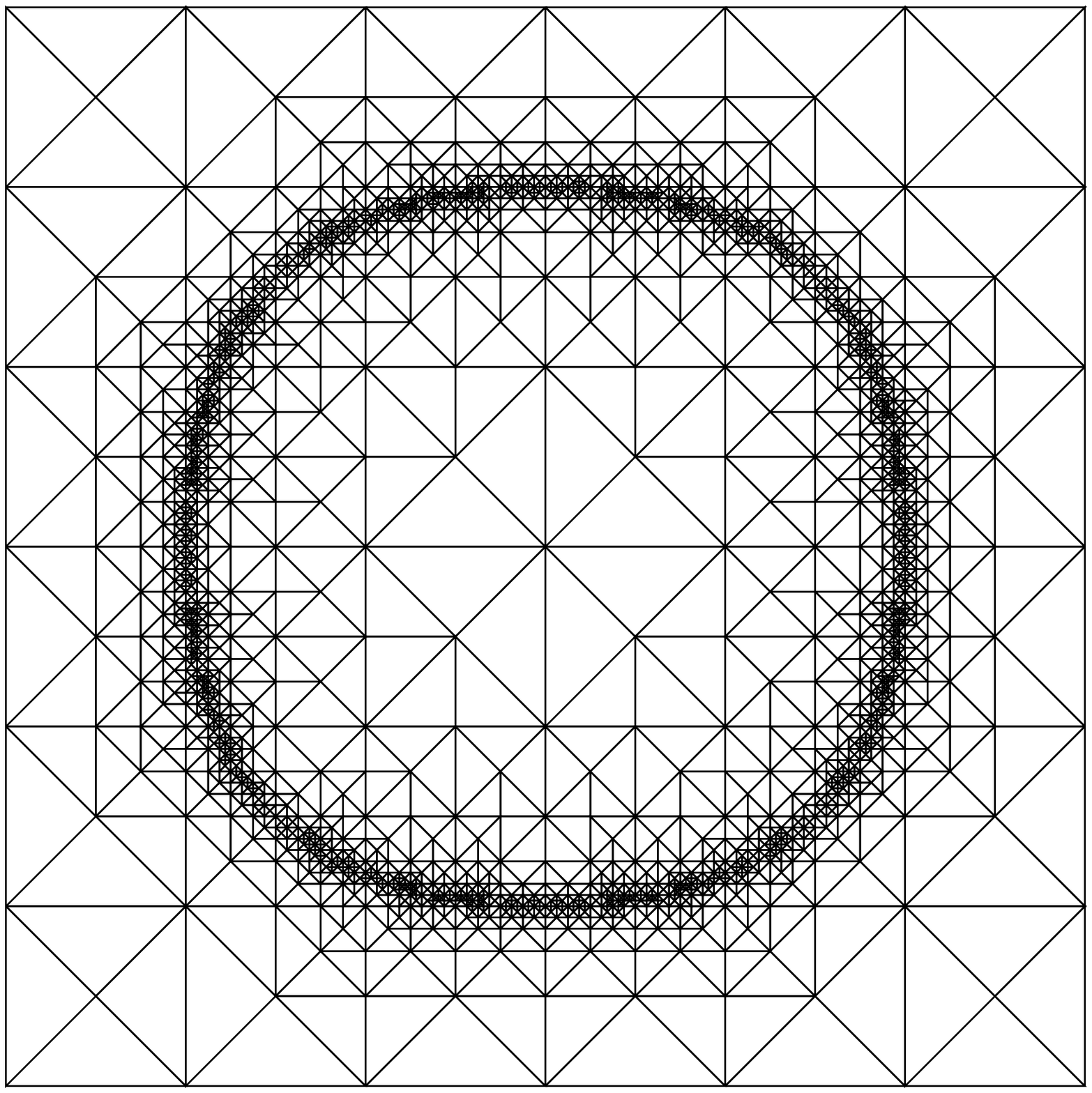}~\hfil~
\includegraphics[trim=90pt 145pt 90pt 145pt, 
clip=true,width=0.4\textwidth ]{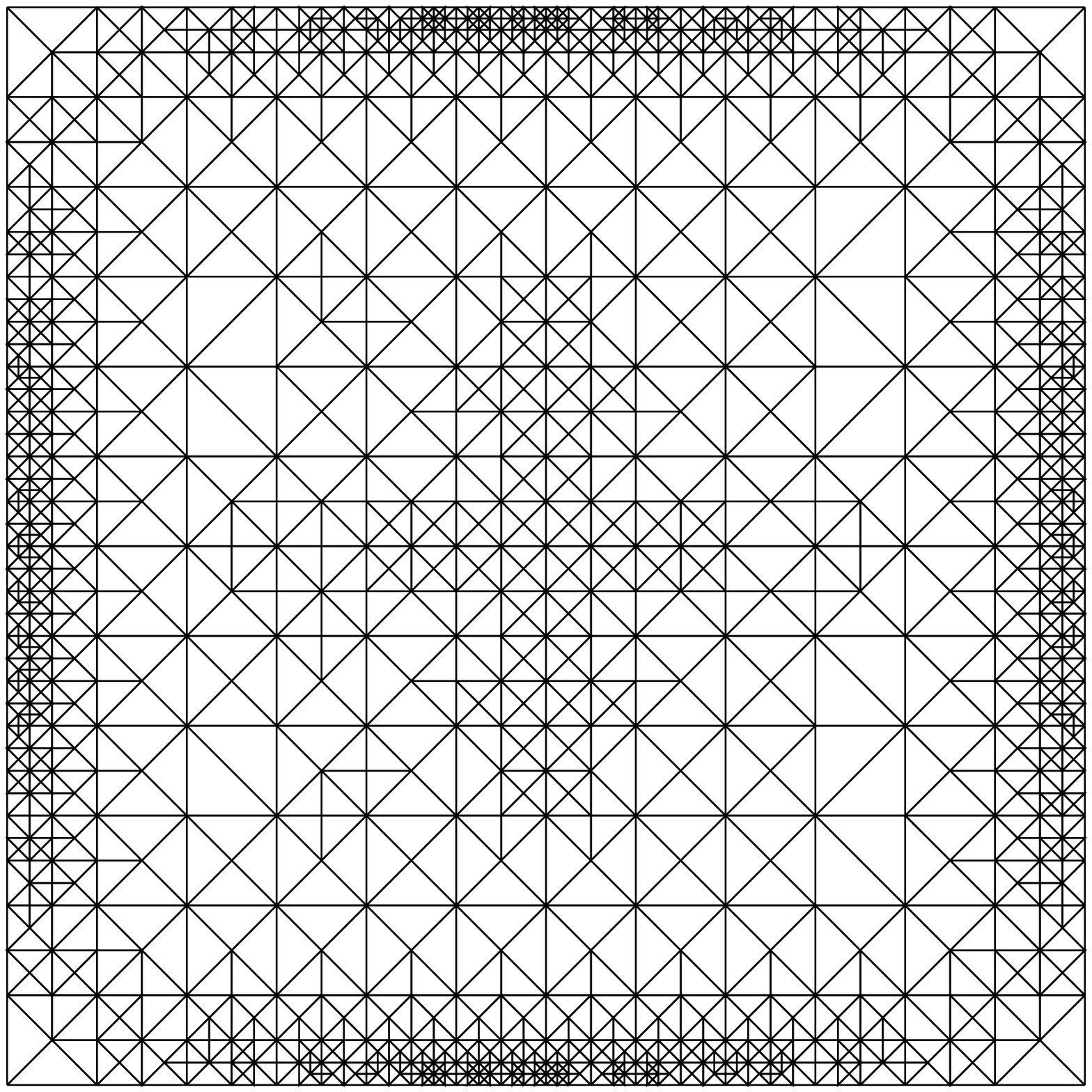}
\caption{Adpative meshes from Example~\ref{example1}.
Left: mesh after 20 iterations with 4140 elements for source~\eqref{ex1:source1}.
Right: mesh after 25 iterations with 2208 elements for source~\eqref{ex1:source2}.
}
\label{fig:ex1mesh}
\end{figure}
The diffusion coefficient $\kappa(s)$ is bounded away from zero with 
$\kappa(s), ~\kappa'(s)$ and $\kappa''(s)$ all bounded assuring 
uniqueness of the solution as in~\cite{CaRa94}. 
Milder versions of this model problem with source~\eqref{ex1:source1} 
are considered by the author in~\cite{Pollock14a,Pollock15a}. 
The inexact iteration~\eqref{eqn:inexactit} allows the solver to stabilize on a
coarser mesh and with steeper gradients in the diffusion term than in previous
versions of the method. Figure~\ref{fig:ex1mesh} on the left shows an adaptive
mesh in the preasymptotic regime for Example~\ref{example1} with source
~\eqref{ex1:source1} showing local refinement of the internal layer with the mesh
remaining coarse away from the layer.  The mesh on the right shows the 
boundary and internal layer resolution for Example~\ref{example1} with 
source~\eqref{ex1:source2}.

Figure~\ref{fig:ex1s1_error_gd} shows the $H^1$ error, 
the $L_2$ error and the error estimator running from the course mesh regime 
through the preasymptotic and into asymptotic regime.  The transition from the 
coarse mesh to preasymptotic is captured by the increase in error within
the first twelve iterations, and through the nonmonotonicity of parameters
$\gamma$ and $\delta$, shown on the right.  The decreases in $\delta$, and
the corresponding increases in $\gamma$ during this phase are due
to reset of the algorithm. The transition from the preasymptotic to the 
asymptotic regime is displayed as a sharp decrease in $L_2$ error and the
estimator $\eta$ as $\delta$ approaches one and the iteration solves for 
a right hand side consistent with the problem data.  This sharp drop can be 
smoothed over several iterations 
by allowing a few preliminary refinements based only on the variable
dependent data or otherwise stalling the initial increase of $\delta$, 
but these techniques were not used in the current work in order to better
illustrate the range of possible behaviors of the inexact method.  
In the plot of the parameters $\gamma$ and $\delta$, 
Figure~\ref{fig:ex1s1_error_gd} also illustrates how 
the initial relationship $\gamma = 1/\delta$
is approximately maintained when the
iterations maintain stability through the mesh refinements.

Figure~\ref{fig:ex1s2_error_gd} shows the behavior of the error estimator
next to a plot of the parameters $\gamma \an \delta$ for the diffusion 
problem of Example~\ref{example1} with a nonsmooth source~\eqref{ex1:source2}.
The initial coarse mesh problem is a bad representation of the exact problem
resulting in early resets of the algorithm, after which enough of the data
is captured to maintain stability through the mesh refinements.  In contrast to 
the same diffusion operator with source~\eqref{ex1:source1}, here the error
estimator grows throughout the preasymptotic regime as $\delta$ increases to one,
then decreases at the predicted rate of $n^{-1/2}$ with respect to degrees of
freedom $n$.
\begin{figure}
\centering
\includegraphics[trim=0pt 0pt 0pt 0pt, clip=true,width=0.45\textwidth]{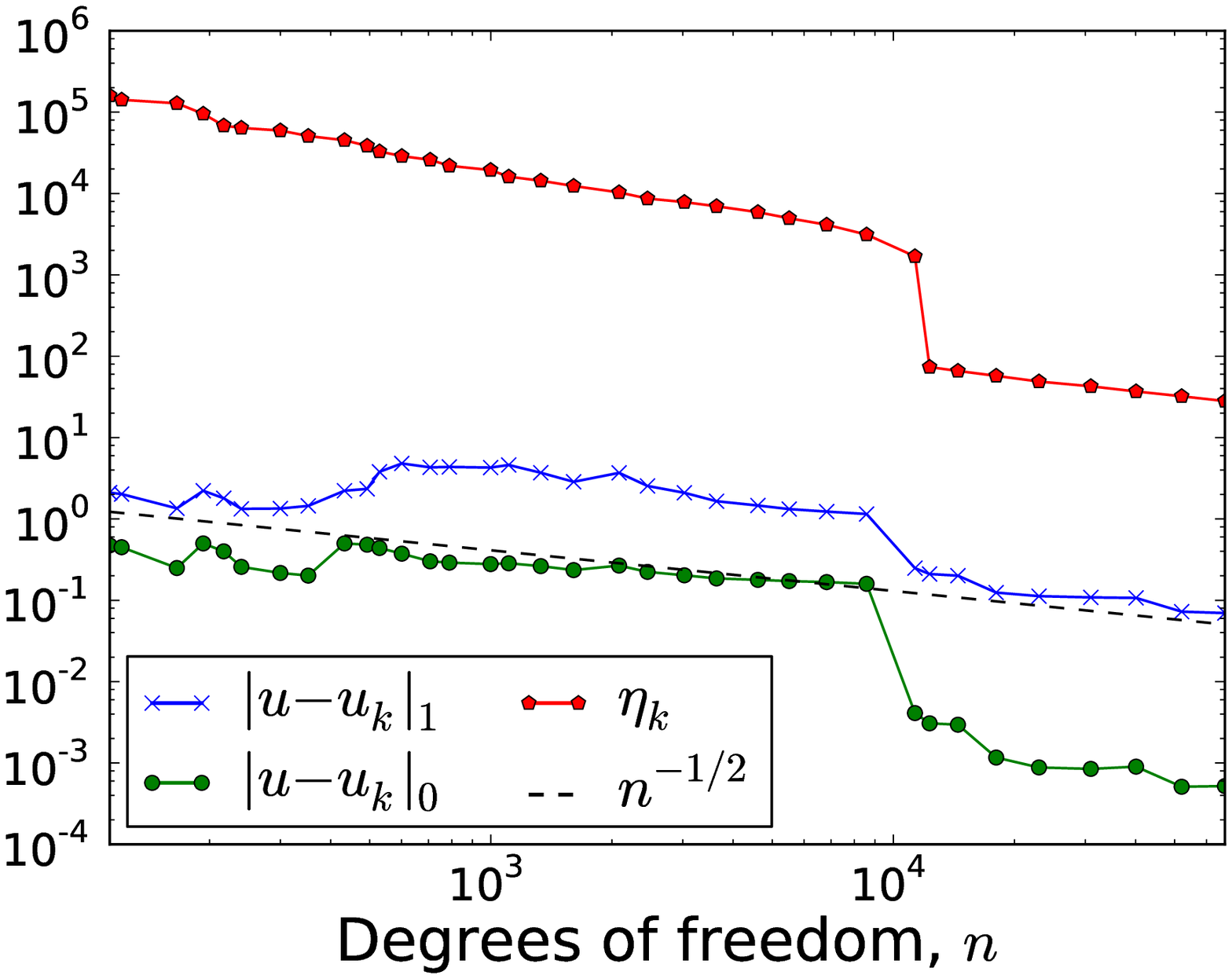}
\includegraphics[trim=0pt 0pt 0pt 0pt, clip=true,width=0.45\textwidth ]{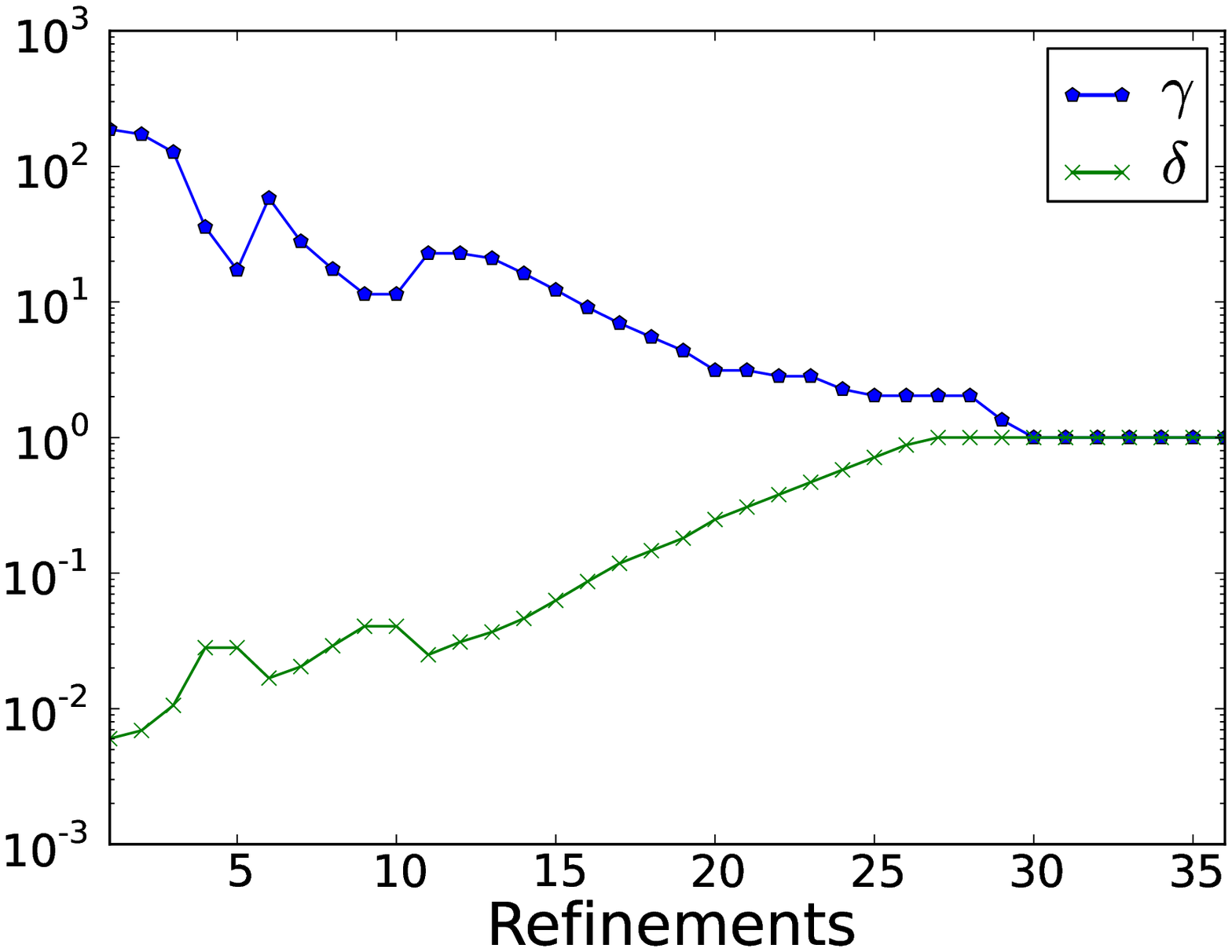}
\caption{Left: $H^1$ error, $L_2$ error,  and error estimator (above) 
against $n^{-1/2}$ where $n$ is the number of degrees of freedom. 
Right: parameters $\gamma$ and $\delta$ 
 for Example~\ref{example1} with source~\eqref{ex1:source1}.}
\label{fig:ex1s1_error_gd}
\end{figure}
\begin{figure}
\centering
\includegraphics[trim=0pt 0pt 0pt 0pt, clip=true,width=0.45\textwidth]{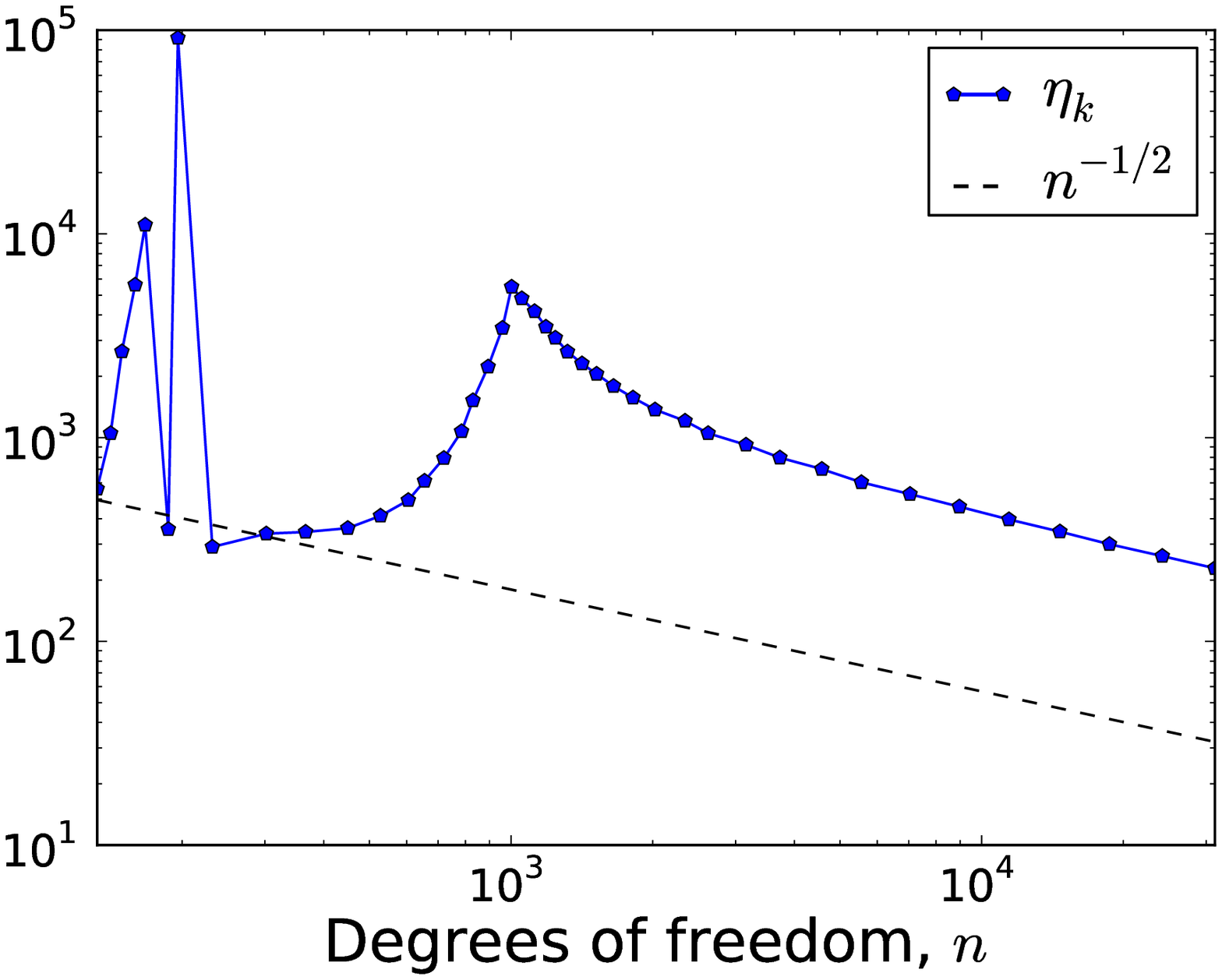}
\includegraphics[trim=0pt 0pt 0pt 0pt, clip=true,width=0.45\textwidth ]{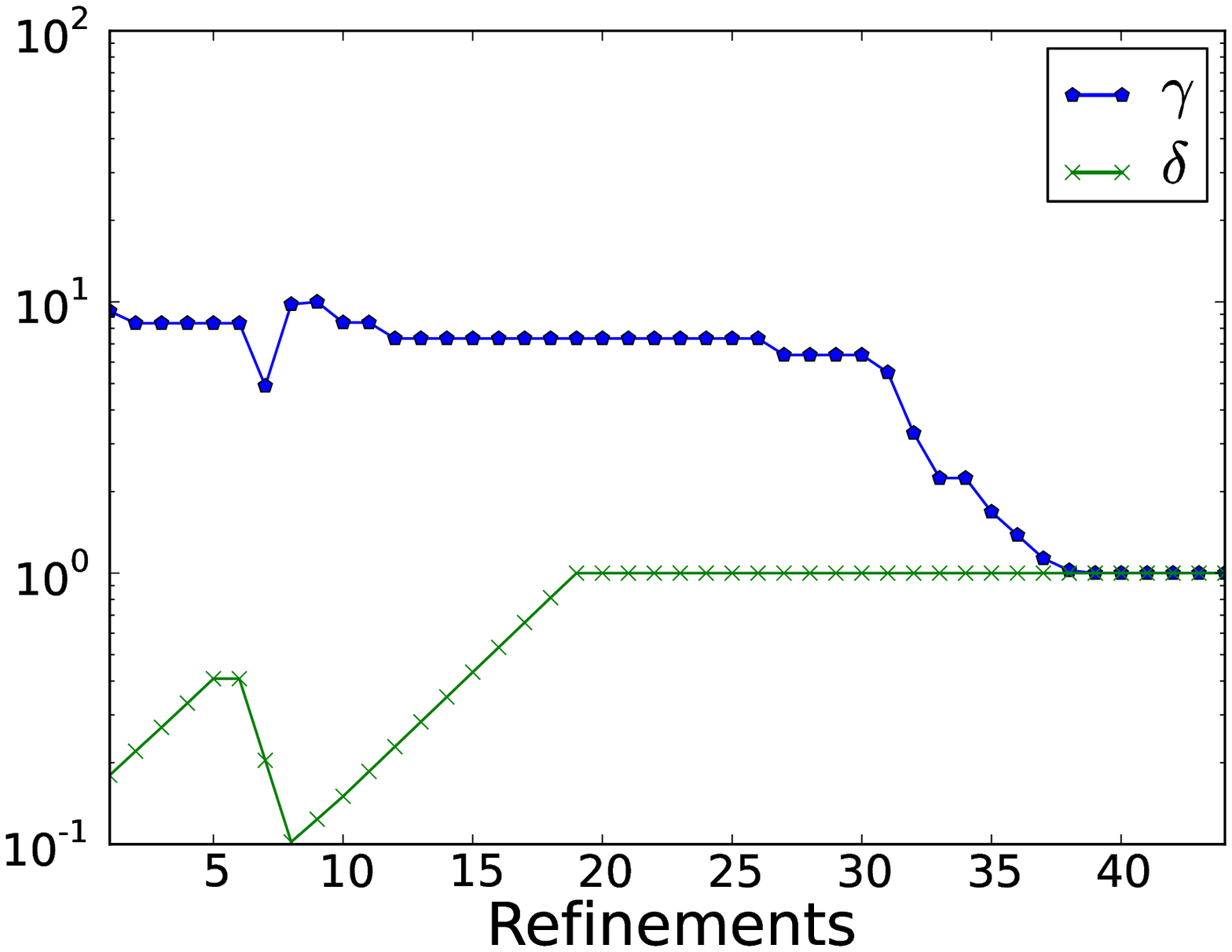}
\caption{Left: Error estimator $\eta$ against $n^{-1/2}$ where $n$ is the number of degrees of freedom. 
Right: parameters $\gamma$ and $\delta$ 
 for Example~\ref{example1} with source~\eqref{ex1:source2}.}
\label{fig:ex1s2_error_gd}
\end{figure}
\begin{example}[Oscillatory diffusion with several layers]\label{example2}
Consider the quasilinear diffusion equation on $\Omega = [0,1]\times[0,1]$.
\begin{align}\label{eqn:ex2}
F(u,x) \coloneqq - \div(\kappa(u) \grad u) - f(x,y) = 0 \inn \Omega, 
~ u = 0 \on \pa \Omega.
\end{align}
The solution dependent diffusion is given by
\begin{align}\label{eqn:ex2kappa}
\kappa(s) = k + \sin(s/\eps) + \arctan(s/\eps), \quad \with 
~\eps = 6\times 10^{-3}.
\end{align}
The following source functions are considered, with the $k$ 
from~\eqref{eqn:ex2kappa} set to $k = 2 + \pi/2$ with source~\eqref{ex2:source1}
and $k = 10$ with source~\eqref{ex2:source2}.
\begin{align}\label{ex2:source1}
f(x,y)& \text{ chosen so }  u(x,y) = 1.85^4\cdot x(x-1)y(y-1)-
2^8\cdot x^{9/4}(x-1)^2y^2(y-1)^2, 
\\ \label{ex2:source2}
f(x,y)& = (1-x)(1-y) (e^{8x^2}-1)(e^{8y^2}-1).
\end{align}
\end{example}
\begin{figure}
\centering
\includegraphics[trim=0pt 0pt 0pt 0pt, clip=true,width=0.45\textwidth]{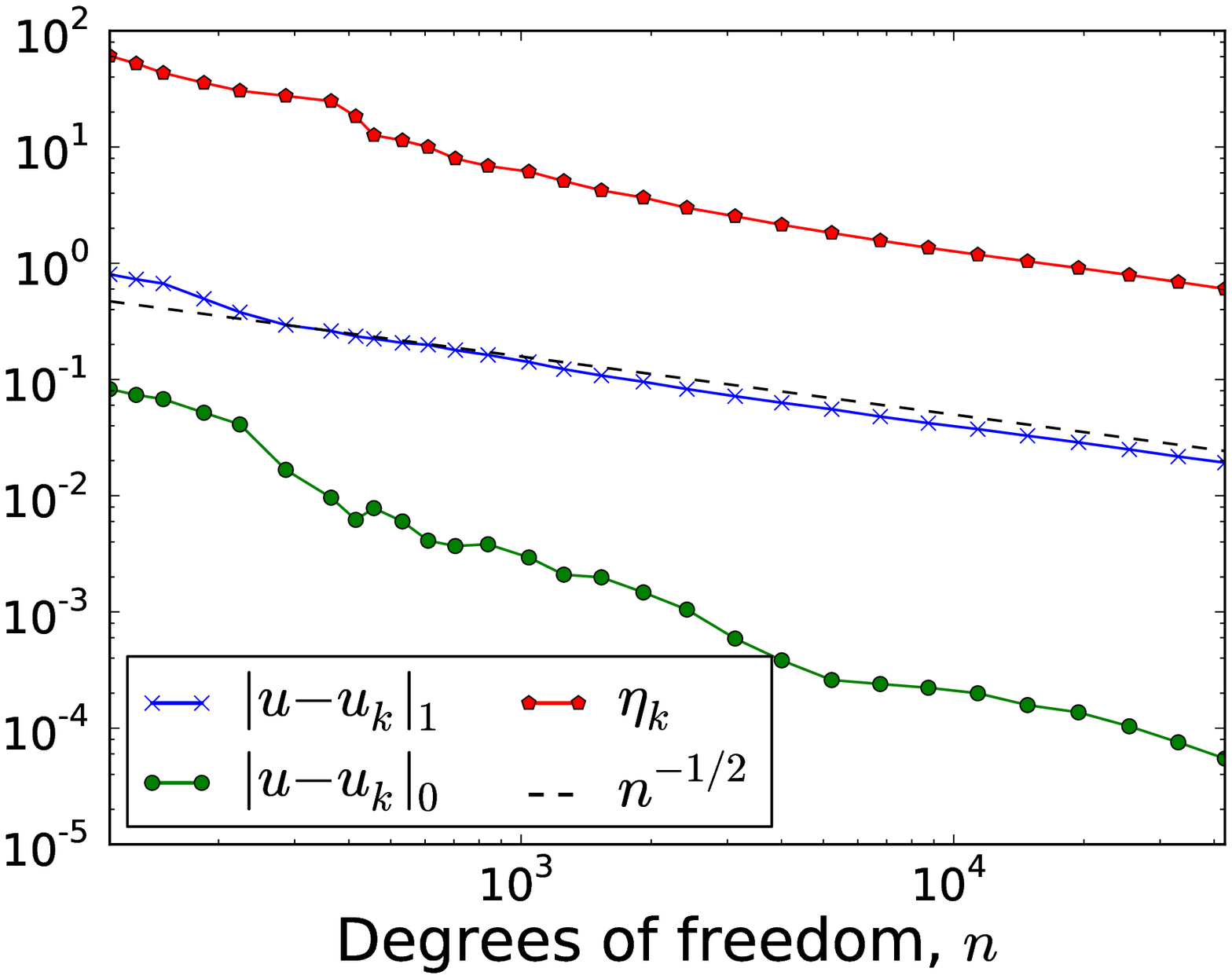}
\includegraphics[trim=0pt 0pt 0pt 0pt, clip=true,width=0.45\textwidth ]{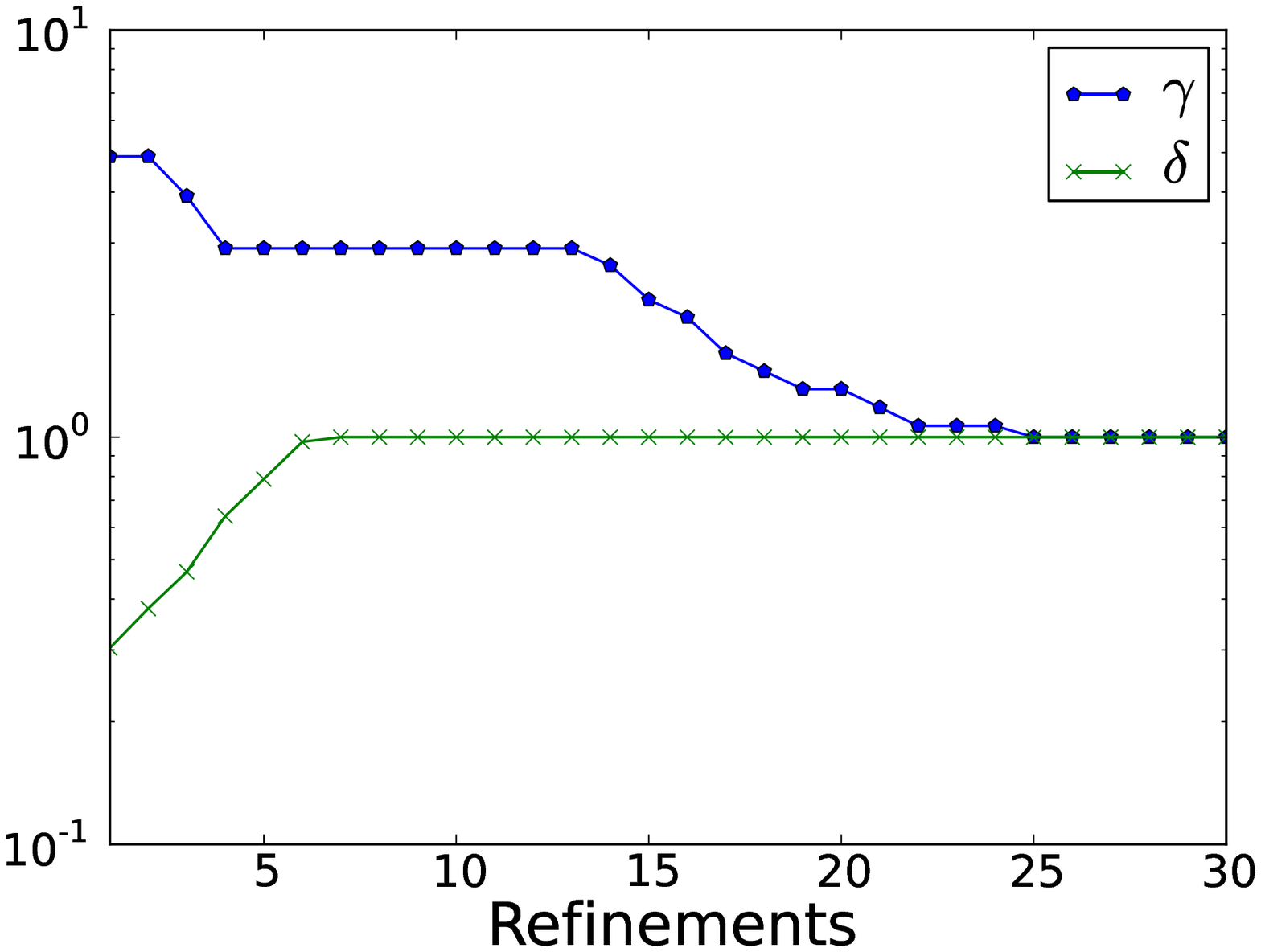}
\caption{Left: $H^1$ error, $L_2$ error and error estimator (above) against $n^{-1/2}$ where $n$ is the number of degrees of freedom. Right: parameters $\gamma$ and $\delta$ 
 for Example~\ref{example2} with source~\eqref{ex2:source1}.}
\label{fig:ex2s1_error_gd}
\end{figure}
\begin{figure}
\centering
\includegraphics[trim=0pt 0pt 0pt 0pt, clip=true,width=0.45\textwidth]{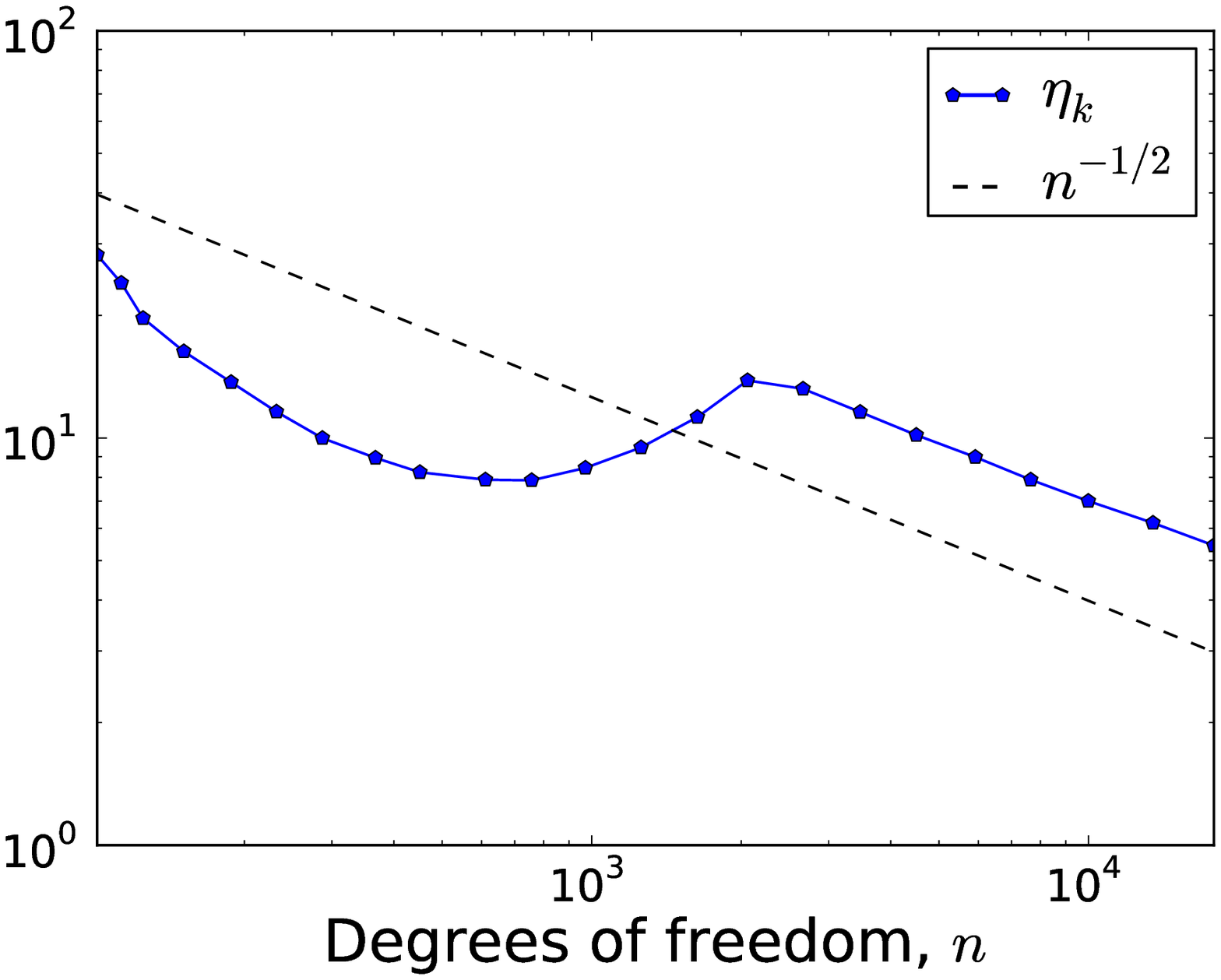}
\includegraphics[trim=0pt 0pt 0pt 0pt, clip=true,width=0.45\textwidth ]{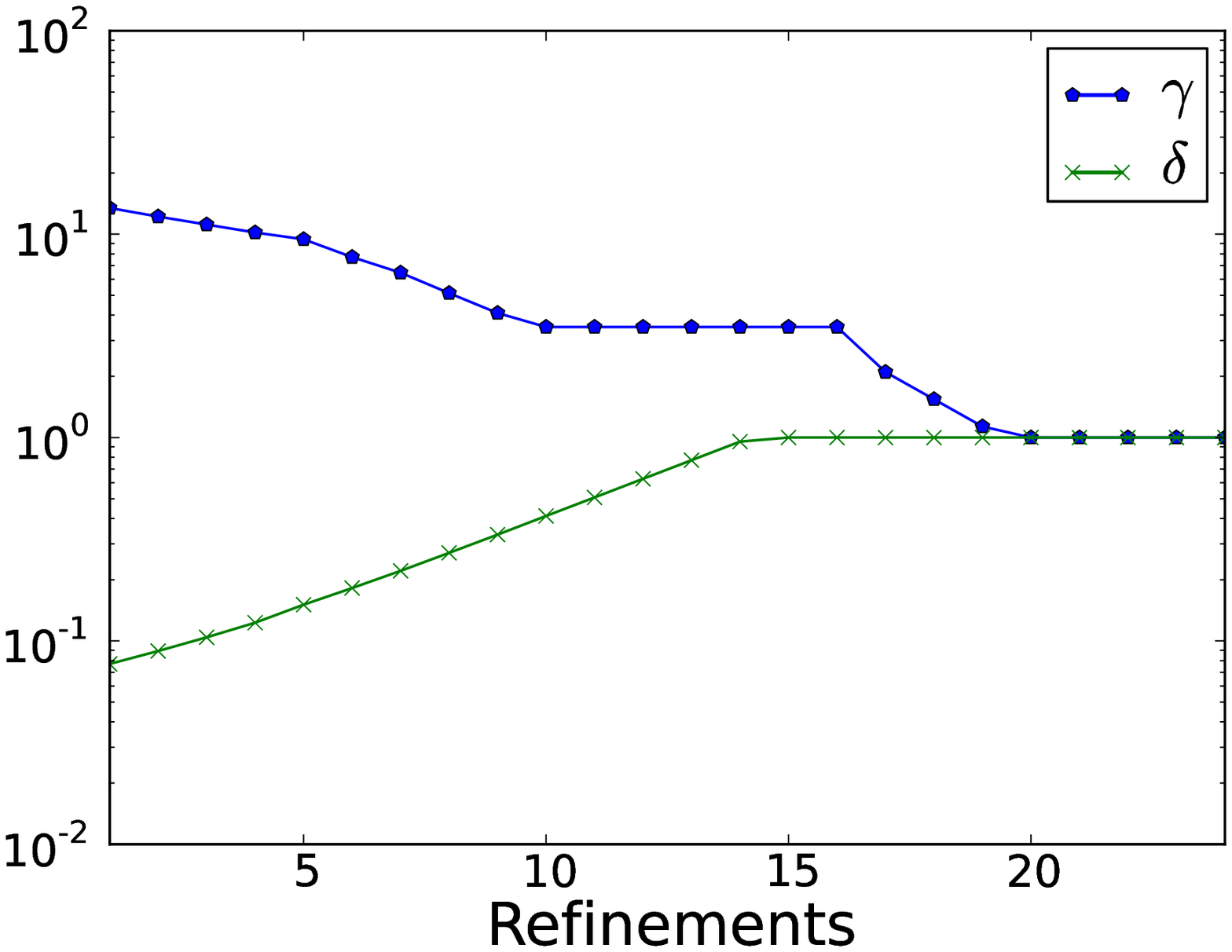}
\caption{Left: Error estimator against $n^{-1/2}$ where $n$ is the number of degrees of freedom. Right: parameters $\gamma$ and $\delta$ 
 for Example~\ref{example2} with source~\eqref{ex2:source2}.}
\label{fig:ex2s2_error_gd}
\end{figure}
By the same reasoning as Example~\ref{example1}, this problem has a unique
solution.
Example~\ref{example2} demonstrates the inexact method on a problem where in
addition to steep gradients, the diffusion coefficient features oscillations
proportional in frequency to the magnitude of solution $u$.  
Similar to the first example
the error decreases throughout all iterations with minor fluctuations for
the source with the known solution~\eqref{ex2:source1}. In this case for the
source corresponding to an unknown solution~\eqref{ex2:source2}, the 
estimator first decreases then stabilizes and 
increases as $\delta$ increases to one, and then decreases at a steady rate.

For both source functions in Example~\ref{example2}, the decrease in $\gamma$
stalls for several iterations in the preasymptotic regime.  This shows
the iterations are not converging at the expected rate for several refinements,
presumably because the
solution-dependent frequencies of the diffusion coefficient $\kappa(u)$,  
are not be sufficiently well expressed on the mesh as the solution $u$ first
attains a threshold frequency. 

\begin{example}[Gradient-dependent diffusion coefficient]\label{example3}
Consider the quasilinear 
diffusion equation on $\Omega = [0,1] \times [0,1]$.
\begin{align}\label{eqn:ex3}
F(u,x) \coloneqq - \div(\kappa(|\grad u|^2) \grad u) - f(x,y) = 0 \inn \Omega, 
~ u = 0 \on \pa \Omega.
\end{align}
The solution dependent diffusion coefficient is given by
\begin{align}
\kappa(t^2) = k + \arctan((t^2 - a)/\eps), \quad \with~ a = \pi, 
~k = \pi, ~\an ~\eps = 2 \times 10^{-2}.
\end{align}
The following source functions are considered
\begin{align}\label{ex3:source1}
f(x,y)& \text{ chosen so }  u(x,y) = \sin (\pi x) \sin (\pi y),  
\\ \label{ex3:source2}
f(x,y)& = 2 \times 10^3 \cdot \left( 0.5 - \left|x - 0.5 \right| \right)^2
\left(0.5 - \left| y - 0.5 \right| \right)^2.
\end{align}
\end{example}
\begin{figure}
\centering
\includegraphics[trim=0pt 0pt 0pt 0pt, 
clip=true,width=0.45\textwidth]{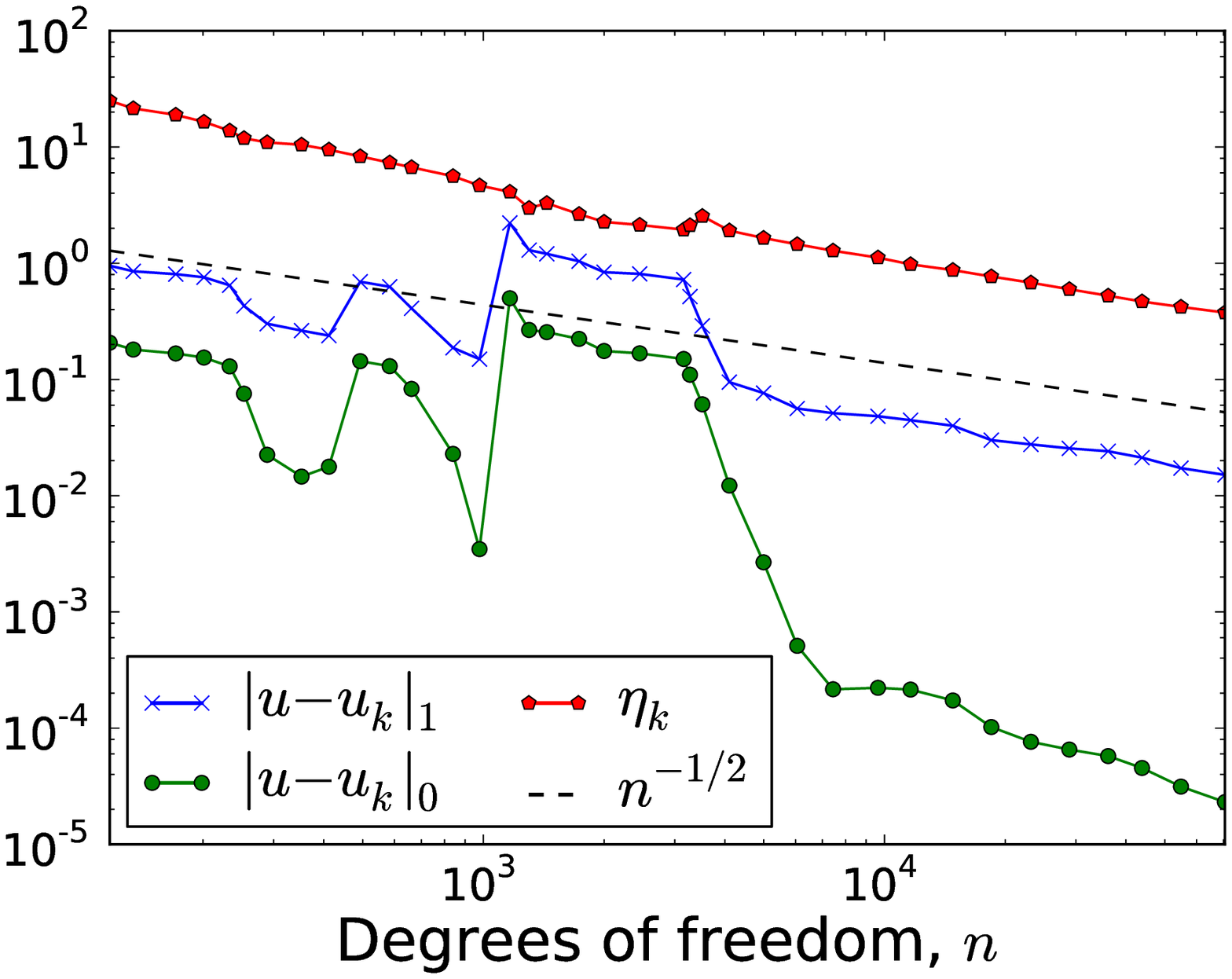}
\includegraphics[trim=0pt 0pt 0pt 0pt, 
clip=true,width=0.45\textwidth ]{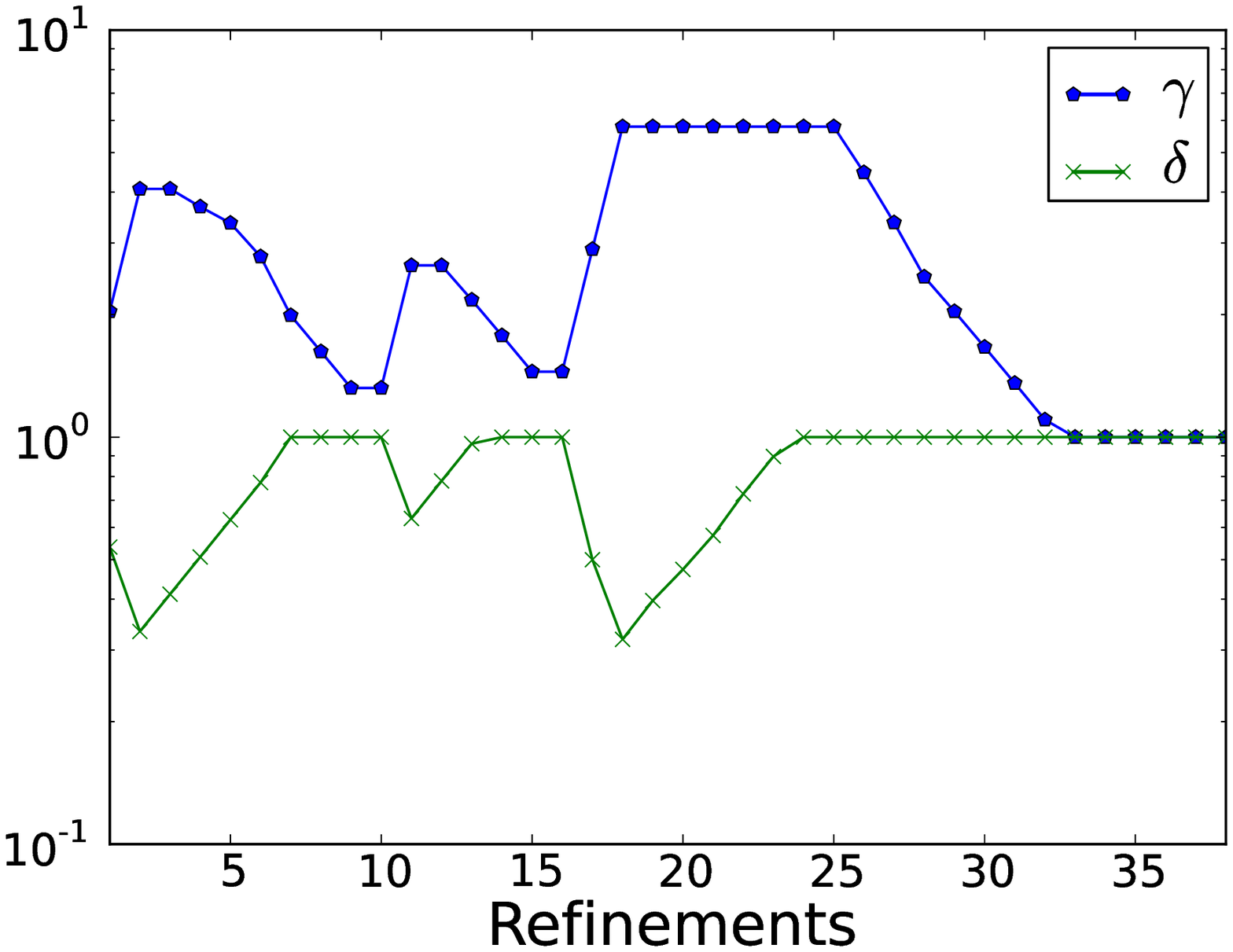}
\caption{Left: $H^1$ error, $L_2$ error and error estimator (above) against 
                $n^{-1/2}$ where $n$ is the number of degrees of freedom. 
         Right: parameters $\gamma$ and $\delta$ 
                for Example~\ref{example3} with source~\eqref{ex3:source1}.}
\label{fig:ex3s1_error_gd}
\end{figure}
\begin{figure}
\centering
\includegraphics[trim=90pt 125pt 90pt 160pt, 
clip=true,width=0.3\textwidth]{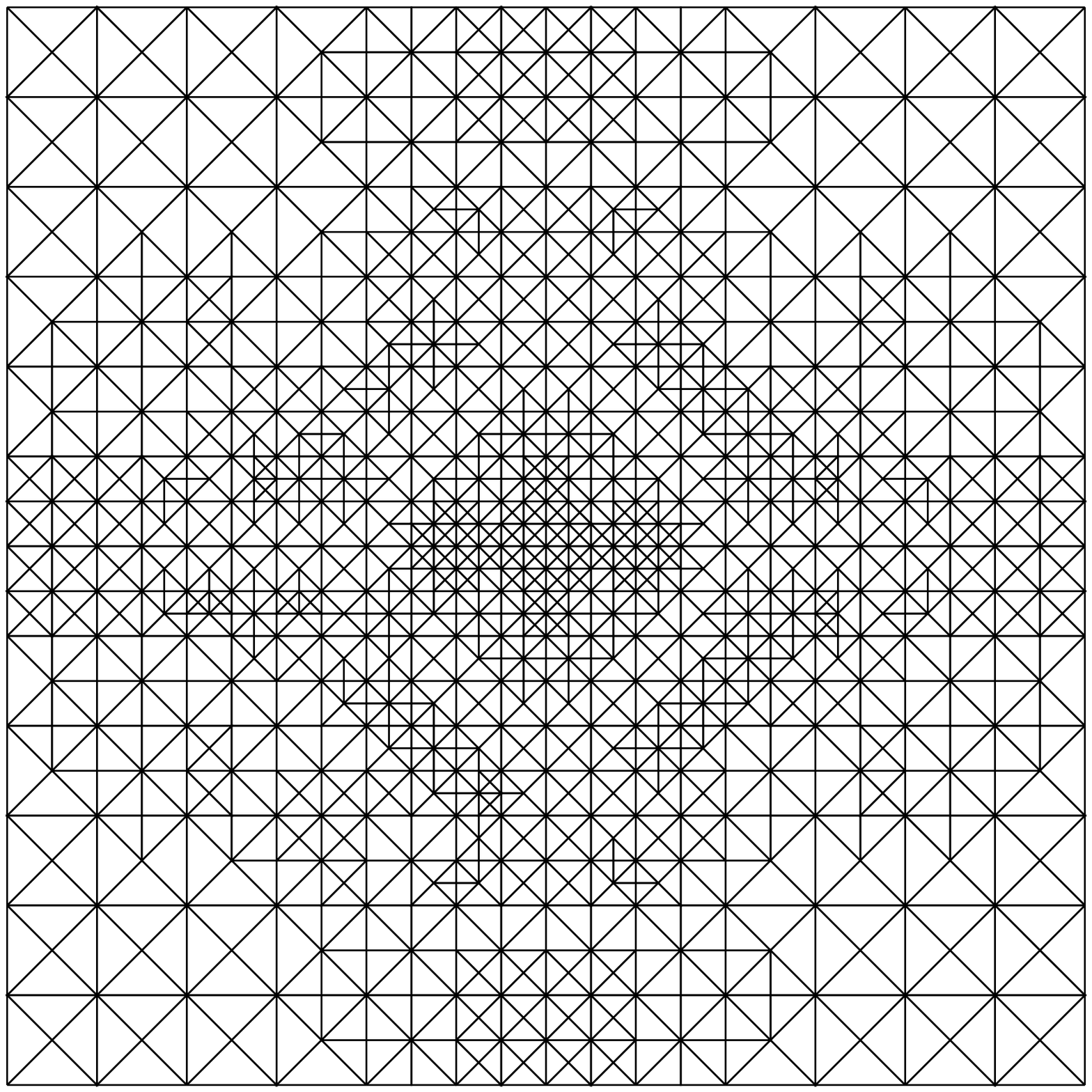}~\hfil~
\includegraphics[trim=0pt 0pt 0pt 0pt, 
clip=true,width=0.45\textwidth ]{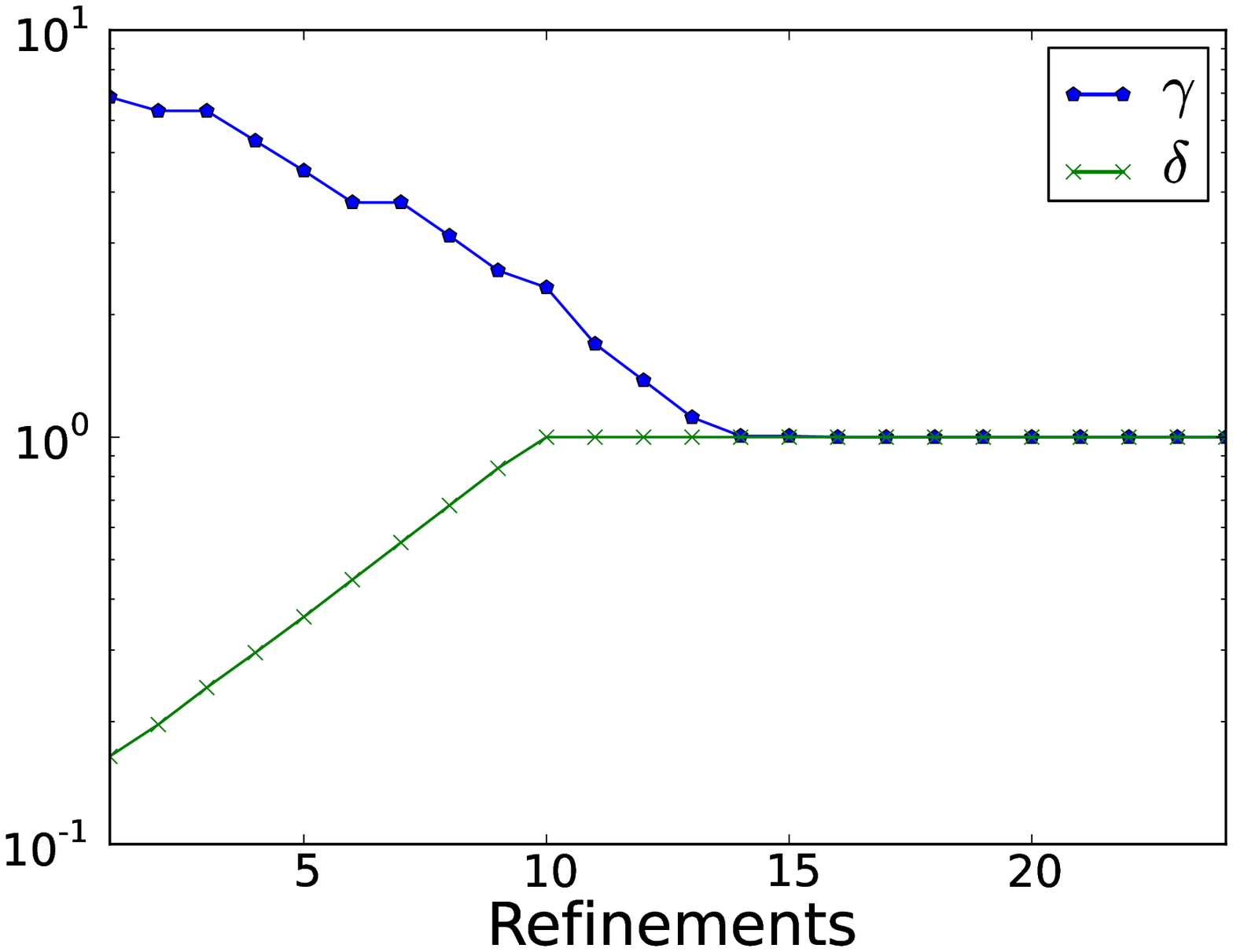}
\caption{Left: adpative mesh from Example~\ref{example3} with source
               ~\eqref{ex3:source2} after
               12 refinements with 1914 elements.
         Right: parameters $\gamma$ and $\delta$.}
\label{fig:ex3s2_mesh_gd}
\end{figure}
As in~\cite{GaMoZu11}, monotonicity and Lipschitz conditions
are established by observing there are 
positive constants $c$ and $C$ with 
$c \le \kappa(t^2) + 2t^2 \kappa'(t^2) \le C$, 
assuring a unique solution to the exact problem. 
This is not known for the approximate problems, particularly 
on the coarse meshes.
This example with source~\eqref{ex3:source1}
illustrates the phenomenon of multiple resets as thin 
layers in both the source $f(x)$ and the Jacobian are uncovered.
Figure~\ref{fig:ex3s1_error_gd} corresponding to Example~\ref{example3} with
source~\eqref{ex3:source1} shows several jumps up in the error corresponding to
the resets, a steep decrease as $\delta \goto 1$ and the discrete problem 
regains consistency, then steady decrease in the error at the predicted rate.
The plot on the right shows $\gamma$ stalling for several refinements after the 
last reset until the approximate Jacobian is sufficiently well resolved 
and Criteria~\ref{assume:gamma} for updating $\gamma$ are met before
Criteria~\ref{criteria:delta} determines the iterations should stop.  
Once necessary resolution is achieved, 
$\gamma$ reduces quickly to one and the residual is solved
to tolerance within a few iterations on subsequent refinements.

The decrease in the estimator is relatively steady in comparison to the 
$H^1$ error, and does not indicate that preliminary to the resets
the correction steps $w^n$ fail to update the solution $u^n$
to sufficiently decrease the residual $\nr{r^n}$. 
In this case the update steps $w^n$ display a loss of stability
due to a qualitative change in the problem data from the latest 
mesh refinement, albeit
at a magnitude small in comparison to $u^n$, and the reset is performed
before the solution is dominated by unstable fluctuations. 

For source~\eqref{ex3:source2} corresponding to the unknown solution, 
the estimator decreases similarly to the one for source~\eqref{ex3:source1}.
The adaptive mesh for the twelfth refinement, just as $\gamma \goto 1$ is 
shown on the left of Figure~\ref{fig:ex3s2_mesh_gd}, next to a plot
of parameters $\gamma$ and $\delta$.  The mesh shows the most dense 
refinement in the center, corresponding to the peak of the source, with additional
refinement along the discontinuities of its gradient.  Additional dense 
refinement is seen developing along the 
steep gradients of $\kappa(|\grad u|^2)$.
\section{Conclusion}\label{sec:conclusion}
An updated adaptive method for the numerical solution of quasilinear diffusion
problems with steep internal layers is presented, building on the
previous work by the author in~\cite{Pollock14a,Pollock15a}.
The goal of this method is to start the computation on a coarse mesh 
where the solution and variable dependent problem coefficients are not resolved,
and to partially solve the possibly ill-posed approximate 
coarse mesh problems up to sufficient 
tolerance to refine the mesh and eventually achieve a stable and efficient
computation of the solution to
an accurate representation of the problem.
The current results include explicit definitions for the added numerical
dissipation parameter $\gamma$ used to stabilize the iterations 
on the coarse mesh, 
and the introduction of an inexact method to further stabilize problems
where the residual is dominated by a variable dependent source. 
The scaling parameter $\delta$ of the inexact method is shown to reach one
which demonstrates consistency of the method, and the convergence of the residual
of the discrete nonlinear problem is established.
Under convergence of the residual, the parameter $\gamma$ is also shown
to reach one, establishing efficiency of the Newton-like iterations in 
the asymptotic regime. 
Future work by the author will include adapting the current method for
quasilinear convection problems and
nonhomogeneous boundary conditions,  as well as investigation of the stabilization 
properties of inexact linear solves. 


\bibliographystyle{siam}
\bibliography{refsTRN}



\end{document}